\documentclass[11pt]{article}
\usepackage[utf8]{inputenc}
\usepackage{amsmath}
\usepackage{amssymb}
\usepackage{amsthm}
\usepackage{amsmath}
\usepackage{mathtools}
\usepackage{bbm}
\usepackage[a4paper,top=2.5cm,bottom=2.5cm,left=2cm,right=2cm,  bindingoffset=5mm, marginparwidth=1.75cm]{geometry}
\usepackage[numbers]{natbib}
\usepackage{graphicx}
\usepackage{dsfont}
\usepackage{textcomp}
\usepackage{url}
\usepackage{sidecap}
\usepackage{array}
\usepackage{chngcntr}
\usepackage{MnSymbol}
\usepackage{wasysym}
\usepackage[hidelinks]{hyperref}
\usepackage{framed}
\usepackage{xcolor} 
\usepackage{transparent}
\usepackage{blindtext}
\usepackage{float}
\usepackage{caption}
\usepackage{subcaption}
\usepackage{authblk}
\usepackage{natbib}
\usepackage{rotating}
\usepackage{nomencl}
\usepackage{paralist}
\usepackage{fancyhdr}
\usepackage{pgf}

\let\emptyset\varnothing

\newtheorem{Thm}{Theorem}[section]
\theoremstyle{definition}
\newtheorem{Def}[Thm]{Definition}
\newtheorem{Ass}[Thm]{Assumption}
\newtheorem{Prop}[Thm]{Proposition}
\newtheorem{Lemma}[Thm]{Lemma}
\newtheorem{remark}[Thm]{Remark}

\providecommand{\keywords}[1]{\textbf{\textbf{Keywords}} #1}
\newcommand{\prob}{\mathbb{P}} 
\newcommand{\E}{\mathbb{E}} 
\newcommand{\N}{\mathbb{N}} 
\newcommand{\Z}{\mathbb{Z}} 
\newcommand{\R}{\mathbb{R}} 
\newcommand{\T}{\mathbb{T}} 
\newcommand{\St}{\mathbb{S}} 
\newcommand{\M}{\mathbb{M}} 
\newcommand{\1}{\mathbbm{1}} 
\newcommand\dist[2]{\lambda_{#2}(#1)} 
\newcommand{\statdist}{\pi} 
\newcommand{\reacdist}{\mu^{AB}} 
\newcommand{\reacdistnorm}{\hat{\mu}^{AB}} 
\newcommand{\trans}{P} 
\newcommand{\back}[1]{#1^-} 
\newcommand{\transback}{\back{\trans}} 
\newcommand{\Mod}[1]{\ (\mathrm{mod}\ #1)} 
\newcommand{\current}{f^{AB}} 
\newcommand{\effcurrent}{f^+} 
\newcommand{\rate}{k^{AB}} 
\newcommand{\rateA}{k^{A\rightarrow}} 
\newcommand{\rateB}{k^{\rightarrow B}} 
\newcommand{\meantime}{t^{AB}} 
\newcommand{\avmeantime}{\bar{t}_N^{AB}} 
\newcommand\restr[2]{\ensuremath{\left.#1\right|_{#2}}}
\newcommand{\PC}{D}

\title{Extending Transition Path Theory:\\  Periodically-Driven and Finite-Time  Dynamics}

 \author[a,b,c]{Luzie Helfmann}
 \author[a,b]{Enric Ribera Borrell}
 \author[a,b]{Christof Sch{\"u}tte}
 \author[a]{P{\'e}ter Koltai}
    
 \affil[a]{Institute of Mathematics, Freie Universit\" at Berlin, Berlin,  Germany}
 \affil[b]{Zuse-Institute Berlin, Berlin, Germany}
 \affil[c]{Department of Complexity Science, Potsdam Institute for Climate Impact Research, Potsdam, Germany}

\date{September 2020}

\begin{document}

\maketitle
\begin{abstract}
Given two distinct subsets $A,B$ in the state space of some dynamical system, 
Transition Path Theory (TPT) was successfully used to describe the statistical behavior of transitions from $A$ to $B$ in the ergodic limit of the stationary system. We derive generalizations of TPT that remove the requirements of stationarity and of the ergodic limit, and provide this powerful tool for the analysis of other dynamical scenarios: periodically forced dynamics and time-dependent finite-time systems. This is partially motivated by studying applications such as climate, ocean, and social dynamics. 
On simple model examples we show how the new tools are able to deliver quantitative understanding about the statistical behavior of such systems. We also point out explicit cases where the more general dynamical regimes show different behaviors to their stationary counterparts, linking these tools directly to bifurcations in non-deterministic systems.
\end{abstract}


\section{Introduction}

The understanding of when and how dynamical transitions, such as tipping processes, happen 
is important for many systems from physics, 
biology~\cite{noe2009constructing}, ecology~\cite{ scheffer2001catastrophic,hastings2018transient},  the climate~\cite{lenton2008tipping,lenton2013environmental} and the social sciences~\cite{nyborg2016social,otto2020social}. 
 
 If the system dynamics can be modelled by a stationary Markov process running for infinite time, Transition Path Theory (TPT) provides a rigorous  approach for studying the transitions from one subset $A$ to another subset $B$ of the state space. 
The main tool of Transition Path Theory~\cite{weinan2006towards, metzner2009transition} are the forward and backward committor probabilities   telling us the probability of the Markov process   to next commit  to (i.e., hit)    $B$  relative to   $A$, either forward or backward in time.
Given these committor probabilities, one can derive important statistics of the ensemble of reactive trajectories (i.e., of  the collection of all possible  paths of the Markov process that start in $A$ and end in $B$), such as
\begin{compactitem}
\item the \emph{density} of reactive trajectories telling us about the bottlenecks during transitions,
\item the \emph{current} of reactive trajectories indicating the  most likely transition channels,
\item the \emph{rate} of reactive trajectories leaving $A$ or entering $B$, and 
\item the \emph{mean duration} of reactive trajectories.
\end{compactitem}

Other approaches that characterize the ensemble of transition paths, on the one hand, place the focus elsewhere: For instance, in Transition Path Sampling~\cite{BolhuisChandlerDellago2002} one is interested in directly sampling trajectories of the reactive ensemble. On the other hand, these approaches
consider different objects, such as the steepest descent path \cite{ulitsky1990new,czerminski1990self,olender1997yet},
the most probable path \cite{olender1996calculation,elber2000temperature, pinski2010transition, Faccioli2010dominant, Beccara2012dominant} (also in temporal networks \cite{ser2015most}), or the first passage path ensemble~\cite{von2018statistical} (see also Remark~\ref{rem:non-ergodic_processes}).

For many physical, especially molecular, systems that are equilibrated and where transitions happen on a smaller time scale than the observation window, the assumption of a stationary, infinite-time Markov process is reasonable and common practice~\cite{noe2009constructing}. 
For illustration we consider the overdamped Langevin dynamics  $$dX_t =  - \nabla V(X_t) dt + \sigma dW_t$$ in the triple well   landscape $V(x,y)$ (as in Figure~\ref{fig:potential_intro}), and we are interested in the transitions between the deep well $A$ and the other deep well~$B$. If the noise intensity $\sigma$ is sufficiently small, then the system tends to spend long times near local minima of deep wells (this behavior is called \emph{metastability}) and transitions predominantly happen across regions of possibly low values of the potential $V$ (e.g., saddle points). If the system from Figure~\ref{fig:potential_intro} is stationary and has infinite time for transitioning, the transition channel via the metastable well centered at $(0,1.5)$ is preferred since the barriers are lower and it does not matter that transitions take a very long time due to being stuck in the metastable set.

\begin{figure}[htb!]
\begin{subfigure}[c]{0.33\textwidth}
\includegraphics[width=1\textwidth]{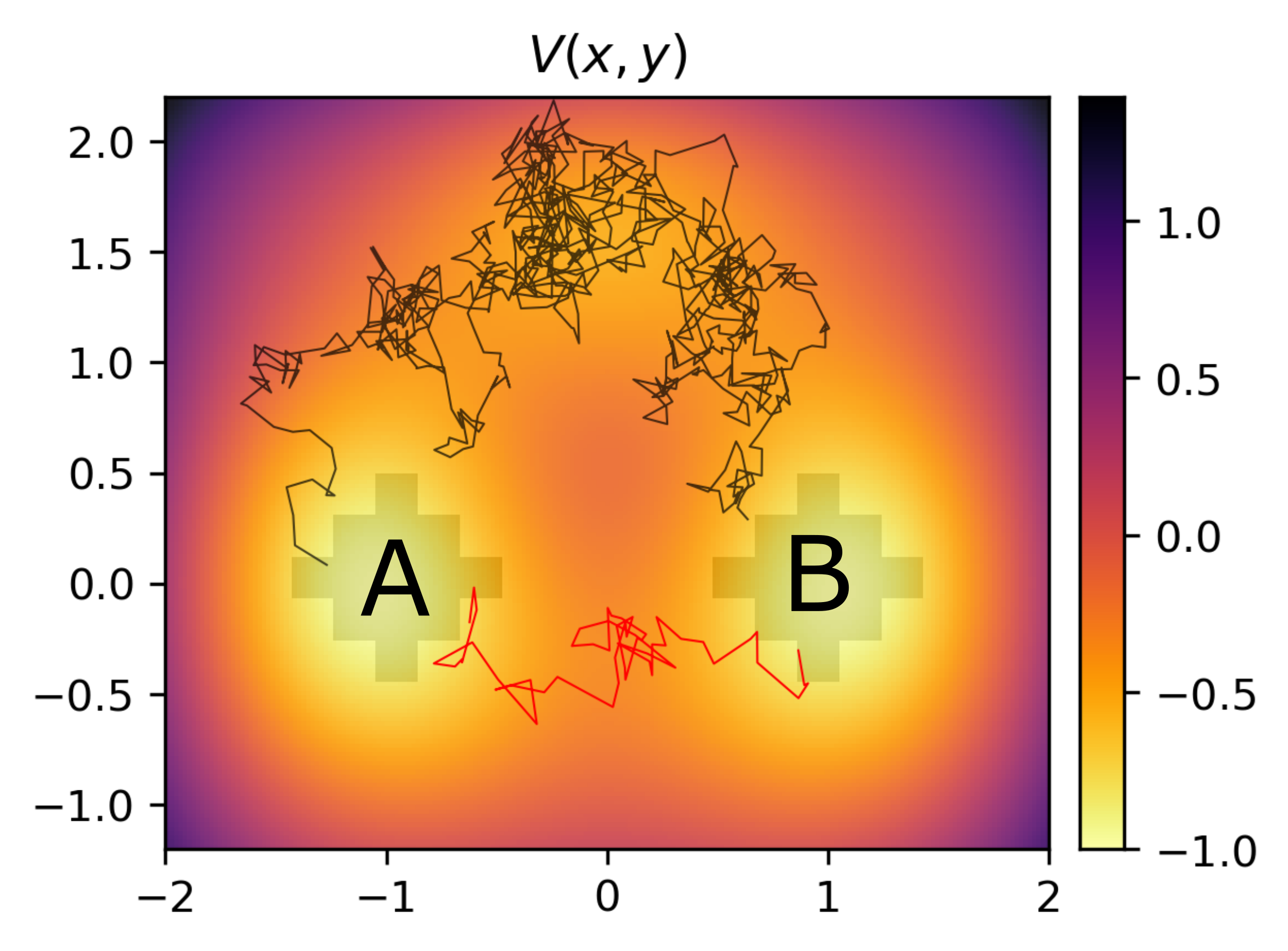}
\caption{}
\end{subfigure}
\begin{subfigure}[c]{0.33\textwidth}
\includegraphics[width=1\textwidth]{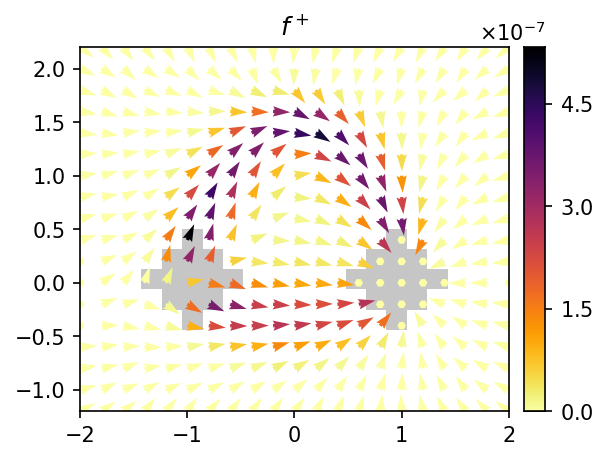}
\caption{}
\end{subfigure}
\begin{subfigure}[c]{0.33\textwidth}
\includegraphics[width=1\textwidth]{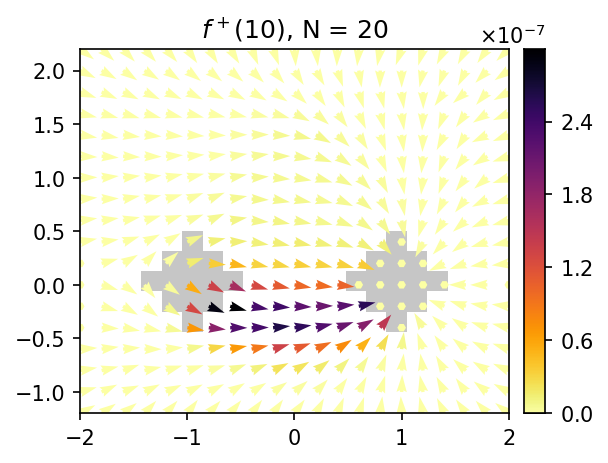}
\caption{}
\end{subfigure}
    \caption{\textbf{(a)} Triple well potential landscape with two possible reactive trajectories, i.e., transitions from $A$ to $B$, highlighted in black and red. \textbf{(b)} The effective current of reactive trajectories from $A$ to $B$  in infinite time; the channel via the well at $(0,1.5)$ is dominantly taken. \textbf{(c)} The effective current of reactive trajectories restricted to a small  time window; the transition channel via the direct barrier between $A$ and $B$ is more likely. More details on the numerical example can be found in Section~\ref{sec:triplewell_ex}, but here we chose a smaller noise intensity $\sigma = 0.26$. }
    \label{fig:potential_intro}
\end{figure}
 
However, in order to study transitions and tipping paths in different dynamical contexts (e.g., social systems or climate models) which are often characterized by  time-dependent (e.g., seasonal) dynamics as well as transitions of interest within a finite time window, the current theory of transition paths has to be extended.
   
In these applications one might, for instance, ask: What are the possible transition channels from the current state to a desirable and sustainable  state of our social or climate system within the next 30 years~\cite{steffen2018trajectories, otto2020social}?  By requiring the transitions to depart from $A$ and arrive in $B$ within a finite time interval, the affinity of the system taking the different transition channels is altered. This is visualized in  the triple well dynamics, Figure~\ref{fig:potential_intro}(c), where  now only the lower transition channel passing the high barrier is possible. Whenever the trajectory takes the channel through the upper metastable set (cf.\ black trajectory in Subfigure (a)), it is stuck there for a long time, and will not reach $B$ anymore within the finite time horizon.
   
Moreover, systems containing human agents are usually time-inhomogeneous and not equilibrated while climate systems are often affected by seasonal forcing, raising questions such as: What are the likely spreading paths of a contagion in a   time-evolving network~\cite{pan2011path,brockmann2013hidden,valdano2018epidemic}?  What are bottlenecks in the transient dynamics towards the equilibrium state~\cite{schonmann1992pattern,hollander2000metastability}?  How does   tipping  occur under the joint effect of noise and parameter changes~\cite{ashwin2012tipping,giorgini2019predicting} or   periodic forcing~\cite{herrmann2005exit}?  
   
In this paper we generalize TPT to a broader class of dynamical scenarios. In particular, we focus on two generalizations which we consider as natural but not exclusive building blocks for these more general cases: (a) periodically forced infinite-time system and (b) arbitrary time-inhomogeneous finite-time system.

We start in Section \ref{sec:basics_of_tpt} by formulating the general setting of TPT for Markov chains $(X_n)_{n\in\T}$ on a finite state space\footnote{Note that we chose Markov chains mostly for simplicity, it is possible to extend the theory to time-continuous and space-continuous dynamics. 
Also, using Ulams method (\cite{ulam1960collection}, see \cite{koltai2011efficient} for a summary)  any continuous Markov system can be discretized into a Markov chain model. }  by introducing time-dependent forward committor functions $q^+_i(n)$, giving the probability within the time horizon $\T$ to next commit to $B $ and not $A$ conditional on being in  $X_n=i$, as well as time-dependent backward committor functions $q^-_i(n)$, both of which are needed for computing the desired statistics of the transitions from $A$ to~$B$. 

We then in more detail work out the following main cases.
\begin{itemize}	
    \item[(i)] Under the assumption of stationary, infinite-time dynamics, it is known   \cite{weinan2006towards,metzner2009transition} that the committor functions are time-independent $q^+(n) = q^+$ by stationarity  and solve the following linear system  
    \begin{equation} 
\left\{ \begin{array}{rcll}
q_i^+ &=& \sum\limits_{j\in \St} \, \trans_{ij} \, q_j^+ &i \in (A\cup B)^c   \\
    q_i^+ &=& 0& i \in A  \\
    q_i^+ &=& 1& i \in B  \\
\end{array}\right.
\end{equation} where $P = (P_{ij})_{i,j\in\St}$ is the transition matrix. Similarly all the transition statistics are time-independent and by ergodicity the statistics can also be found by averaging along one infinitely long equilibrium trajectory~\cite{weinan2006towards,metzner2009transition}.   
In Section~\ref{sec:tpt_ergodic} we recall the theory~\cite{weinan2006towards,metzner2009transition} from a different point of view. Instead of defining all quantities in terms of trajectory-wise time-averages, we  prove  by using the Markov property that they can be written in terms of the committors and the stationary distribution. Further, we extend the theory by Lemma~\ref{Lm:q_fb_all_paths} decomposing the committors into path probabilities.
    \item[(ii)]  In Section~\ref{sec:tpt_periodic} we derive the committors and transition statistics for periodically varying dynamics with a period of length $M$ that are equilibrated (i.e., the law of the chain is also periodic). It follows that the committors are periodically varying $q^+(n) = q^+_m$ whenever $n=m$ modulo $M$, and we show that the committors solve the following linear system with periodic boundary conditions  in time $q^+_0 = q^+_M$
    \begin{equation} 
\left\{ \begin{array}{rcll}
q_{m,i}^+ &=& \sum\limits_{j\in \St} \, \trans_{m,ij} \, q_{m+1,j}^+ & i \in (A\cup B)^c \\
    q_{m,i}^+ &=& 0& i \in A  \\
    q_{m,i}^+ &=& 1& i \in B  \\
\end{array}\right.
\end{equation}
where $P_m = (P_{m,ij})_{i,j\in\St}$ is the transition matrix at time $n=m$ modulo $M$.
    This is consistent with the previous case when choosing a period of length $M=1$.  
    \item[(iii)] In Section~\ref{sec:tpt_finite} we derive the committor equations and transition statistics for general time-inhomogeneous Markov chains $(X_n)_{n\in\T}$ on a finite-time interval $\T=\{0,\dots,N-1\}$ defined by the transition probabilities $P(n) = (P_{ij}(n))_{i,j\in\St}$ and an initial density $\dist{0}{}$.
The forward committor $q^+(n)$ for a finite-time Markov chain  satisfies the following iterative system of equations:
\begin{equation}  
\left\{ \begin{array}{rcll}
q_i^+(n) &=&  
\sum\limits_{j \in \St} \, \trans_{ij}(n) \, q_j^+(n+1)    & i \in (A\cup B)^c  \\
q_i^+(n) &=& 0 & i \in A \\
q_i^+(n) &=& 1 & i \in B
\end{array}\right.
\end{equation}
with final condition $q^+_i(N-1) = \1_B(i)$. The transition statistics depend on the current time point in the time interval and can be related to an average over the ensemble of reactive trajectories.  We  also show consistency,
i.e., that given a stationary process on a finite time interval, the committors and statistics converge to their classical counterparts (i) in the infinite time limit.
\end{itemize}

We note that in the stationary regime, committor functions have been used for finding basins of attraction  of stochastic dynamics~\cite{Kol11,KoVo14,lindner2019stochastic},  as reaction coordinates~\cite{lu2014exact},  as basis functions in the core-set approach~ \cite{schutte2011markov, sarich2011projected,schutte2013metastability} and for studying modules and flows in networks~\cite{djurdjevac2011random,cameron2014flows}. We expect our results to enable similar uses for the time-dependent regime.  

The theoretical results from this paper are accompanied by   numerical studies  on two toy examples,   a network of $5$ nodes with time-varying transition probabilities, and a discrete model of the overdamped Langevin dynamics in a triple well potential with time-dependent forcing (as in Figure \ref{fig:potential_intro}). 
In these examples we will particularly show  to what extent time-dependent dynamics or finite-time restrictions affect the transition statistics in contrast to those in stationary, infinite time dynamics. The TPT-related objects derived here allow for a \emph{quantitative} assessment of the dominant statistical behavior in complicated dynamical regimes:

\begin{itemize}
    \item[(i)] By adding a periodic forcing to the stationary dynamics, the reaction channels are perturbed and new transition paths, that were not possible before, can appear  (see \hyperref[sub:network_periodic]{Example 2} and \hyperref[sub:triplewell_ex_periodic]{7}). 
    \item[(ii)] By restricting the stationary dynamics with matrix $P$ to a finite-time window (cf. \hyperref[sub:network_finite_stationary]{Example 3} and \hyperref[sub:triplewell_ex_finite_hom]{8}), only transitions within this window are allowed and the average rate of transitions is much lower  than in the infinite-time situation. By additionally applying a forcing  (cf. \hyperref[sub:network_finite_inhom]{Example 4}), the system is not equilibrated anymore and we can get a higher average rate of transitions than without forcing, although we set the time-dependent transition matrix $P(n)$ such that  its time-average equals the transition  matrix $P$ of the stationary case. A similar approach is used in rare events simulation where the system is pushed by an optimal non-equilibrium forcing under which the rare events become more likely~\cite{Hartmann2012optimal, Hartmann2014characterization}.
    \item[(iii)] The finite-time case can also be employed for studying qualitative changes in the transition dynamics when parameters are perturbed.  In \hyperref[sub:well_ex_bif]{Example 9}, we exemplarily show  that by increasing the finite-time interval length $N$, the transition dynamics change qualitatively, even though the dynamics  are stationary. Thus, TPT can be used as a quantitative tool describing \emph{bifurcations} in non-stationary non-deterministic systems.
 \end{itemize}
The code used for the examples is available on Github at  \url{www.github.com/LuzieH/pytpt}.

 
\section{Preliminaries and General Setup} \label{sec:basics_of_tpt}

The objective of Transition Path Theory (TPT) is to understand the mechanisms by which noise-induced transitions of a Markov chain $(X_n)_{n\in\T}$ from one subset of the state space $A \subset \St$ to another subset $B \subset \St$ take place\footnote{TPT can be generalized to consider transitions between $N$ subsets $A_1,A_2,\dots, A_N$ by looking at the transitions between $A=A_i$ and $B=\bigcup_{j\neq i}^{N} A_j $ for each $i=1,\dots,N$ (e.g., used in the core set approach \cite{schutte2011markov,sarich2011projected,schutte2013metastability}). }. $A$ and $B$ are chosen as two non-empty, disjoint subsets of the finite state space $\St$,  such that the transition region $C:= (A \cup B)^c$ is also non-empty. 
Since historically in TPT one thinks of $A$  as the reactant states of a system, $B$ as the product states and the transitions from $A$ to $B$  as  reaction events, we call  the pieces of a trajectory that  connect  $A$ to $B$ by the name   \emph{reactive trajectories}. Each reactive trajectory contains the successive states in $C$ visited during a transition  that starts in $A$ and ends in $B$, see Figure \ref{fig:ensemble_reactive}, and we are interested in  gathering dynamical information about  the ensemble of reactive trajectories. 
\begin{SCfigure}[1]
    \centering
    \includegraphics[width=0.35\textwidth]{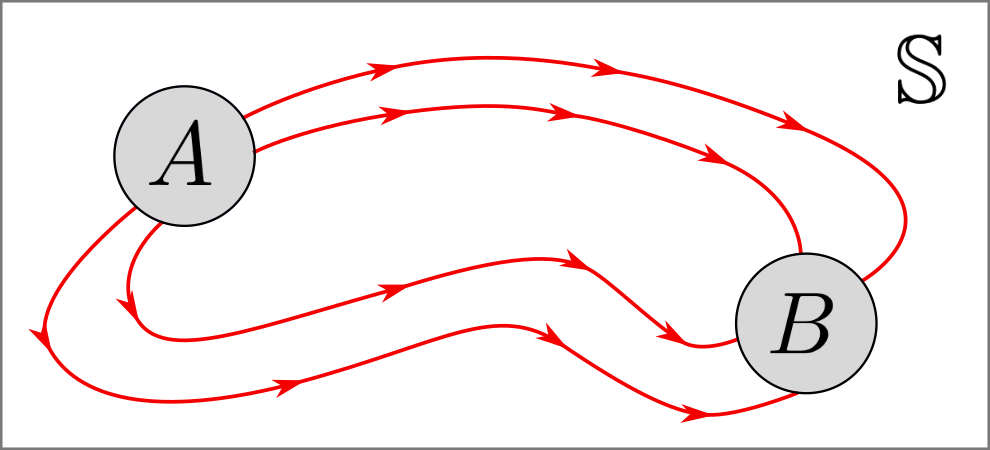}
    \caption{Reactive trajectories, i.e., excursions of trajectories that start in $A$ and end in $B$. }
    \label{fig:ensemble_reactive}
\end{SCfigure}

In this paper,  we are  interested not only in Markov chains living on the infinite time frame~$\T = \Z$ but also those on finite time intervals $\T=\{0,1,\dots, N-1\}$, that is where the transitions from $A$ to $B$ have to take place during a finite time frame~$\T$. Moreover, we also consider non-stationary dynamics. This either means
\begin{compactitem}
\item that the system is in the transient phase towards equilibrium but otherwise has a time-independent transition matrix, or
\item that the dynamical rules (the transition matrices) are varying in time, as well as
\item that we are dealing with time-varying sets $A$ and $B$ (see the comment in Section \ref{sec:com_periodic} for the case of periodic dynamics and Remark \ref{rem:space_time_set_finite} for finite-time dynamics).
\end{compactitem}

In this section,  we will define the committor functions  and show how they can be used to derive  important statistics of the ensemble of reactive trajectories, this entails, e.g., the frequency of transitions, the most important transition channels and where the process on the way to $B$ gets stuck or spends most of its time. Thereby we will keep everything general enough for time-dependent and finite-time dynamics and for the moment only need to assume the Markovianity of the chain $(X_n)_{n\in\T}$   and that the distribution $\prob(X_n=i)$ and time-dependent transition rule $P(n) = (\prob(X_{n+1}=j|X_n=i))_{i,j\in\St}$ is given for all $n\in \T$.

Later, the results from this section will be applied to the special cases of (i) infinite-time, stationary dynamics (Section \ref{sec:tpt_ergodic}), (ii) infinite-time, periodic dynamics (Section \ref{sec:tpt_periodic}) and (iii) finite-time, time-dependent systems (Section \ref{sec:tpt_finite}), where we will also prove the existence of committors and give the linear system of equations they are solving. These three are just some selected special cases, other interesting cases of systems are, e.g., stochastic regime-switching (see the Remark \ref{rem:switching}). 

\subsection{Committor Equations}
All of the transition statistics and characteristics can be computed from  the committor probabilities, therefore we will start by defining them. The \emph{forward committor} is the probability that the Markov chain, currently in some state $i$, will next\footnote{ where with next we also include the current time point, i.e., that the system is already in $B$} go to set $B$ and not to $A$. The \emph{backward committor} is the same for the time-reversed Markov chain, i.e., the probability that the time-reversed process will next hit $A$ and not $B$, or equivalently, the probability that the chain last came from $A$ and not $B$. 

More precisely, the forward committor $q^+(n)=(q^+_i(n))_{i\in\St}$   gives the probability that the process starting in $i\in\St$ at time $n\in\T$ reaches at next within $\T$ first $B$ and not $A$. We write
\begin{equation}  \label{eq:q_f_def} \begin{split} 
q_i^+(n) := \prob(\tau_B^+(n) < \tau_A^+(n) \,|\, X_n=i), 
\end{split} \end{equation}
where the  \emph{first entrance time} of a set $S\subset \St$ after or at time $n\in\T$ is given by
$$\tau_S^+(n) := \inf\{k\in\T \text{ s.t. } k \geq n,  X_k \in S \}, \ \inf \emptyset := \infty.$$  
The backward committor $q^-(n)=(q^-_i(n))_{i\in\St}$ 
gives the probability that the trajectory arriving  at time $n\in\T$ in state $i\in\St$ last within $\T$ came from $A$ not $B$,
\begin{equation} \label{eq:q_b_def}  \begin{split}
q_i^-(n) :=  \prob( \tau_A^-(n) >\tau_B^-(n) \,|\, X_n=i), 
\end{split} \end{equation}
where the  \emph{last exit time} of a set $S$ before time or at time $n$ is given by 
$$\tau_S^- (n) := \sup\{k \in \T \text{ s.t. }k \leq n, X_k \in S \},\ \sup \emptyset := - \infty.$$
Note that the first entrance and last exit times are stopping times with respect to the forward and time-reversed process \cite[Lemma 3.1.1]{RiberaBorrell2019}. The time-reversed process will be introduced in the following sections for   infinite-time and  finite-time Markov chains.

The forward committor, as it is defined, only considers trajectory pieces arriving at $B$ within the time horizon $\T$, similarly, the backward committor only considers excursions that left $A$ within~$\T$.

\subsection{Transition Statistics}
\label{sec:statistics}

We will now define statistical objects that characterize the ensemble of \emph{reactive trajectories} from subset $A $ to $B  $  and see that they can be computed using the  committor probabilities and the  Markovianity assumption.

The first two objects, the  distribution of reactive trajectories  $\reacdist(n)$ and its normalized version $\reacdistnorm(n)$, tell us where the reactive trajectories are most likely to be found, i.e., where the transitioning trajectories spend most of their time.
\begin{Def}\label{def_reacdist} The
\textit{distribution of reactive trajectories} $\reacdist(n) =(\reacdist_i(n))_{i\in \St}$ for $n\in\T$ gives  the joint probability that the Markov chain is in a state $i$ at time $n$ while transitioning from $A$ to $B$:
 
    $$\reacdist_i(n) := \prob(X_n=i, \tau_A^-(n) >\tau_B^-(n) , \tau_B^+(n) < \tau_A^+(n)).$$
\end{Def}
Note that $\reacdist_i(n)=0$ for $i \notin C$, i.e., we only get information about the density of transitions passing through $C$. Direct transitions from $A$ to $B$ are neglected, and in general assumed not to exist.
\begin{Thm}\label{thm_reacdist}  For a general Markov chain $(X_n)_{n\in\T}$ with committors $q^+(n)$, $q^-(n)$, the distribution of reactive trajectories can be expressed as
    \begin{equation}  \begin{split}
    \reacdist_i(n)& =q_i^-(n)  \, \prob(X_n=i) \, q_i^+ (n) . \nonumber
\end{split} \end{equation}
\end{Thm}
\begin{proof} We can compute
\begin{equation}  \begin{split}
    \reacdist_i(n)&=\prob(\tau_A^-(n) >\tau_B^-(n) , \tau_B^+(n) < \tau_A^+(n)\,|\, X_n=i) \,  \prob(X_n=i)
    =q_i^-(n)\, q_i^+ (n) \, \prob(X_n=i) \nonumber
\end{split} \end{equation}
by conditioning on $\{X_n=i\}$, and by using  independence of the two events $\{\tau_A^-(n) >\tau_B^-(n) \}$, $\{\tau_B^+(n) < \tau_A^+(n)\}$ given $\{X_n=i\}$, which follows from the Markov property.\footnote{\cite[Prop 2.1.10]{RiberaBorrell2019} provides us with a generalisation of the Markov property for Markov chains for events like $\{\tau_A^-(n) >\tau_B^-(n) \}$ resp. $\{\tau_B^+(n) < \tau_A^+(n)\}$, which belong to the $\sigma$-algebra that contains the present and the future resp. the  and past of the chain.}
\end{proof}
The distribution $\reacdist(n)$ is not normalized but can easily be normalized by dividing $\reacdist(n)$ by  the probability to be on a transition at time $n$, $$Z^{AB}(n) := \sum_{j \in C} \reacdist_j(n) =\prob(\tau_A^-(n) >\tau_B^-(n) , \tau_B^+(n) < \tau_A^+(n)),$$  to give a probability distribution on $\St$:
\begin{Def} Whenever  $\prob(\tau_A^-(n) >\tau_B^-(n) , \tau_B^+(n) < \tau_A^+(n))>0$ for $n\in\T$,  we can define the \textit{normalized distribution of reactive trajectories} at time $n\in\T$ by
$$\reacdistnorm_i(n) := \prob(X_n=i|\tau_A^-(n) >\tau_B^-(n) , \tau_B^+(n) < \tau_A^+(n))$$
giving the density of states in which trajectories transitioning from $A$ to $B$ spend their time.
\end{Def}
The next object tells us about the average number of jumps from $i$ to $j$ during one time step (i.e., the probability flux) while the trajectory is on its way from $A$ to $B$:
\begin{Def}\label{def_reaccurr} 
The \textit{current of reactive trajectories} $\current(n) = (\current_{ij}(n))_{i,j\in \St}$ at time $n\in\T$ gives the average flux of trajectories going through $i$ at time $n\in \T$ and $j$ at time $n+1\in\T$ consecutively while on their way from $A$ to $B$:
$$\current_{ij}(n):= \prob( X_n=i, X_{n+1}=j, \tau_A^-(n) >\tau_B^-(n) , \tau_B^+(n+1) < \tau_A^+(n+1))$$
\end{Def}
\begin{Thm}\label{thm_reaccurr}  The current of reactive trajectories for a Markov chain $(X_n)_{n\in\T}$ with transition probabilities $\trans(n)$ and committors $q^+(n)$, $q^-(n)$, is given by 
\begin{equation}  \begin{split}
    \current_{ij} (n) =   q_i^-(n) \, \prob(X_n=i) \, \trans_{ij}(n) \,  q^+_j(n+1) \nonumber
\end{split} \end{equation}
\end{Thm}
\begin{proof} The reactive current can be computed as
\begin{equation}  \begin{split}
    \current_{ij} (n)
    &=\prob( X_{n+1}=j, \tau_A^-(n) >\tau_B^-(n) , \tau_B^+(n+1) < \tau_A^+(n+1)| X_n=i) \, \prob(X_n=i) \nonumber \\
    &=\prob( X_{n+1}=j,\tau_B^+(n+1) < \tau_A^+(n+1)| X_n=i) \, \prob(\tau_A^-(n) >\tau_B^-(n)| X_n=i) \, \prob(X_n=i)   \nonumber \\
    &=\prob( X_{n+1}=j| X_n=i) \, \prob(  \tau_B^+(n+1) < \tau_A^+(n+1)| X_{n+1}=j, X_n=i)\,  q_i^-(n)\, \prob(X_n=i)  \nonumber \\ 
        &=   q_i^-(n)\, \prob(X_n=i)\, \trans_{ij}(n) \, q^+_j (n+1),\nonumber
\end{split} \end{equation}
by first conditioning on $\{X_n=i\}$, then by independence of $\{X_{n+1}=j, \tau_B^+(n+1) < \tau_A^+(n+1)\}$ and $\{ \tau_A^-(n) >\tau_B^-(n)\}$ given $\{X_n=i\}$,  by conditioning on $\{X_{n+1}=j\}$, and last by the Markov property.
\end{proof}
Let us note that the reactive current also counts the transitions going directly from $i\in A$ to $j\in B$, these are not accounted for in the reactive distribution which only accounts for transitions passing the region  $(A\cup B)^c$.

In order to eliminate information about detours of reactive trajectories, we define:
\begin{Def}\label{def_eff} 
The \textit{effective current of reactive trajectories} $\effcurrent(n) = (\effcurrent_{ij}(n))_{i,j\in \St}$ at time $n\in\T$ gives the net amount of reactive current  going through $i$ at time $n\in \T$ and $j$ at time $n+1\in\T$ consecutively,
$$\effcurrent_{ij}(n):= \max\{\current_{ij}(n)-\current_{ji}(n),0\}.$$
\end{Def}
Ultimately, the effective current of reactive trajectories can be used to find the dominant transition channels in state space between $A$ and $B$ (see, e.g., \cite{metzner2009transition}).

The current of reactive trajectories only goes out of $A$, not into $A$, moreover the current of reactive trajectories only points into $B$, not out of $B$. Therefore, $A$ can be thought of as a source of reactive trajectories, whereas $B$ acts like their sink. This leads us to our next characteristic of reactive trajectories: by summing the current of reactive trajectories over $A$ we get the the  discrete rate of reactive trajectories flowing out of $A$, and by summing the current over $B$, we obtain the rate of inflow into $B$:
\begin{Def}
For $n, n+1 \in \T$, the \textit{discrete rate of transitions leaving $A$} at time $n$ is given by 
$$\rateA (n):= \prob(X_n\in A, \tau_B^+(n+1) < \tau_A^+(n+1)), $$ i.e., the probability of a reactive trajectory leaving $A$ at time $n$. When $n-1, n \in \T$, the \textit{discrete rate of transitions entering $B$} at time $n$ is given by 
$$\rateB (n):= \prob(X_n\in B,\tau_A^-(n-1) >\tau_B^-(n-1)),   $$
 i.e., the probability of a reactive  trajectory entering $B$ at time $n$. 
\end{Def}
\begin{Thm} For a Markov chain $(X_n)_{n\in\T}$ with current of reactive trajectories $\current(n)$, we find the discrete rates to be
\begin{equation}
\begin{split}
    \rateA(n) &= \sum_{i\in A, j\in\St} \current_{ij}(n) \\ 
    \rateB(n) &= \sum_{i\in \St, j\in B} \current_{ij}(n-1). \\ 
\end{split}
\end{equation}
\end{Thm}
\begin{proof}
We can compute by using the law of total probability  
\begin{equation}
    \begin{split}
    \sum_{i\in A, j\in\St} \current_{ij}(n) &= \sum_{i\in A, j\in\St} \prob( X_n=i, X_{n+1}=j, \tau_A^-(n) >\tau_B^-(n) , \tau_B^+(n+1) < \tau_A^+(n+1)) 
    = \rateA(n) \\ 
    \sum_{i\in \St, j\in B} \current_{ij}(n-1) &=\sum_{i\in \St, j\in B} \prob( X_{n-1}=i, X_{n}=j,  \tau_A^-(n-1) >\tau_B^-(n-1) , \tau_B^+(n) < \tau_A^+(n))
    = \rateB(n).         
    \end{split}
\end{equation}
\end{proof}

\section{TPT for Stationary, Infinite-time Markov Chains} \label{sec:tpt_ergodic}
 TPT was originally designed for stationary, infinite-time Markov processes~\cite{weinan2006towards,metzner2009transition},   that are often used as models for  molecular systems~\cite{noe2009constructing}. Here we will recall that theory by using the results from the previous section and equip it with some new results (e.g. Lemma~\ref{Lm:q_fb_all_paths}) that will be needed later.
 
\subsection{Setting}
We begin with describing the processes of interest in this section.
\Ass{We consider a Markov chain $(X_n)_{n\in\Z}$ taking values in a  discrete  and finite state space $\St$, the time-discrete jumps between  states $i\in\St$ and $j\in \St$ occur with   probability $$\trans_{ij} := \prob(X_{n+1}=j|X_n=i)$$ stored in the row-stochastic transition matrix $\trans=(\trans_{ij})_{i,j \in \St}$.
    We assume that the process is irreducible, and ergodic with respect to the unique, strictly positive invariant distribution  $\statdist = (\statdist_i)_{i \in \St}$ (also called stationary distribution interchangeably) solving $\statdist^\top = \statdist^\top \trans $.} \label{ass:ergodic}
    
 The time-reversed process $(\back{X}_n)_{n \in \Z}$, $\back{X}_n := X_{-n}$ traverses the chain backwards in time. It is also a Markov chain \cite[Thm 2.1.19]{RiberaBorrell2019} and stationary with respect to the same invariant distribution. The  transition probabilities of the time-reversed process $\transback = (\transback_{ij})_{i,j \in \St}$ with entries $$\transback_{ij} := \prob(\back{X}_{-n+1} =j |\back{X}_{-n}=i) =  \prob(X_{n-1}=j |X_{n}=i)$$ can be found from expressing the flux in two ways, $$\prob(X_n=i, X_{n+1}=j)= \trans_{ij} \statdist_i = \transback_{ji} \statdist_j.$$
 
\subsection{Committor Probabilities}
Due to stationarity of the chain (Assumption \ref{ass:ergodic}), the law of the chain is the same for all times, and we simply have that the committors \eqref{eq:q_f_def} and \eqref{eq:q_b_def} are   time-independent $q_i^+(n)= q_i^+$, similarly $q_i^-(n)=q_i^-$ for all $n$.

The forward and backward committors can be found by solving a linear matrix equation of size $|C|$ with appropriate boundary conditions \cite[Chapter 1.3]{norris1998markov}. 
\begin{Thm} \label{Thm:q_equ} The forward committor for a   Markov chain according to Assumption \ref{ass:ergodic} with transition probabilities $\trans=(\trans_{ij})_{i,j\in\St}$ satisfies the following linear system 
\begin{equation}\label{eq:q_f}  
\left\{ \begin{array}{rcll}
q_i^+ &=& \sum\limits_{j\in \St} \, \trans_{ij} \, q_j^+ &i \in C   \\
    q_i^+ &=& 0& i \in A  \\
    q_i^+ &=& 1& i \in B.  \\
\end{array}\right.
\end{equation}
Analogously for the  backward committor, we have to solve the following linear system 
\begin{equation}\label{eq:q_b} 
\left\{ \begin{array}{rcll}
q_i^- &=& \sum\limits_{j\in \St} \, \transback_{ij} \, q_j^- &i \in C   \\
    q_i^- &=& 0& i \in B  \\
    q_i^- &=& 1& i \in A.  \\
\end{array}\right.
\end{equation}
\end{Thm}
\begin{proof} From the definition of the committors \eqref{eq:q_f_def}, it immediately follows that we have $ q_i^+ =0 $ for $i \in A$ since we always have $\tau_A^+(n)=n$, while $\tau_B^+(n)>n$. Analogously we have $q_i^+ = 1$ for $i \in B$ since in that case $\tau_A^+(n)>n$ and $\tau_B^+(n)=n$. For the committor at node $i$ in the transition region $C$, we can sum the forward committor at all the other states $j$ weighted with the transition probability to transition from $i$ to $j$. This follows from 
\begin{equation}  \begin{split} 
q_i^+ &=\prob(\tau_B^+(n) < \tau_A^+(n) | X_n=i) = \sum_{j\in \St} \prob(X_{n+1}=j,  \tau_{B}^+(n) <\tau_{A}^+(n) |X_n=i) \nonumber \\ 
&= \sum_{j\in \St} \prob(X_{n+1}=j|X_n=i) \prob(\tau_{B}^+(n) <\tau_{A}^+(n) |X_{n+1}=j,X_n=i) \nonumber \\
&= \sum_{j\in \St} \prob(X_{n+1}=j|X_n=i) \prob(\tau_{B}^+(n+1) <\tau_{A}^+(n+1) |X_{n+1}=j,X_n=i) = \sum_{j\in \St}\trans_{ij} q_j^+ \nonumber
\end{split} \end{equation}
first using the law of total probability, then conditioning on $\{X_{n+1}=j\}$, using that at time $n$ the chain is in $i\in C$ and thus $\tau_A^+(n),\tau_B^+(n)\geq n+1$, and last using the Markov property. 

For the backward committor equations we can proceed in a similar way, by additionally using the time-reversed process.
\end{proof}
\begin{remark}
\label{rem:reversible}
If the Markov chain in addition is reversible, i.e., if  $\trans_{ij} \statdist_i = \trans_{ji} \statdist_j$ (equivalently, $\transback_{ij}=\trans_{ij}$) holds, then  it follows from Theorem \ref{Thm:q_equ}  that the forward and backward committor are related by $q_i^+=1-q_i^-$.
\end{remark}
The following lemma provides us with the necessary condition such that existence and uniqueness of the committors is guaranteed. For a proof see \cite[Lemma 3.2.4]{RiberaBorrell2019} or \cite[Chapter 4.2]{norris1998markov}.
\begin{Lemma} \label{Lm:ex_uniq_ergodic}
If $\trans$ is irreducible, then the two problems \eqref{eq:q_f} and \eqref{eq:q_b} each have a unique solution.
\end{Lemma}
There is a second characterization of the committors in the transition region using path probabilities, as summarized in the following lemma. The proof can be found in the Appendix \ref{app:proofs}.
\begin{Lemma} \label{Lm:q_fb_all_paths}
For any $i \in C$, the forward committor can also be specified as the probability of all possible paths starting from node $i$ that reach the set $B$ before $A$
\begin{equation} \label{eq:q_f_all_paths} \begin{split}
q_i^+ &= \sum_{\tau \in \Z^+} \sum_{\substack{i_1 \dots i_{\tau -1} \in C \\ i_\tau \in B}} \trans_{ii_1} \cdots \trans_{i_{\tau -1} i_{\tau}}.  
\end{split} \end{equation}
Similarly for any $i \in C$, the backward committor can also be understood as the sum of path probabilities of all possible paths arriving at node $i$ that  last came from $A$ and not $B$
\begin{equation} \label{eq:q_b_all_paths} \begin{split}
q_i^- &= \sum_{\tau \in \Z^-} \sum_{\substack{i_\tau \in A\\ i_{\tau +1} \dots i_{-1} \in C  }} \transback_{i_{\tau +1}i_{\tau}} \cdots \transback_{i i_{-1}}.
\end{split} \end{equation}
\end{Lemma}

\subsection{Transition Statistics}
The committor, the distribution,  and the transition probabilities are time-independent, thus the statistics from  Section \ref{sec:statistics} are time-independent. We write the distribution of reactive trajectories (Theorem \ref{thm_reacdist}) as $\reacdist=(\reacdist_i)_{i\in\St}$, where
$$\reacdist_i =q_i^- q_i^+ \ \statdist_i, 
$$
the  normalized distribution as $\reacdistnorm=(\reacdistnorm_i)_{i\in\St}$.
The current of reactive trajectories (Theorem \ref{thm_reaccurr}) $\current=(\current_{ij})_{i,j\in\St}$ is given by 
$$\current_{ij}  = q_i^- \statdist_i \trans_{ij} q^+_j,
$$
and the effective reactive current is denoted $\effcurrent = (\effcurrent_{ij})_{i,j\in \St}$.
\begin{Thm}\label{Thm:conservationlaw} For a stationary Markov chain $(X_n)_{n\in \Z}$ the reactive current out of a node $i\in C$ equals the current flowing into the node $i\in C$, i.e.,
\begin{equation}
 \sum_{j\in\St} \current_{ij} = \sum_{j\in\St} \current_{ji}. 
\end{equation}
Further, the the reactive current flowing out of A into $\St$ (equivalently into $C\cup B$) equals the flow of reactive trajectories from $\St$ (equivalently from $C \cup A$) into $B$ 
\begin{equation}
\sum_{i\in A, j\in\St} \current_{ij} = \sum_{ i\in\St, j\in B} \current_{ij}.
\end{equation} 
\end{Thm}
For the proof see the Appendix \ref{app:proofs}.
\remark{Due to these conservation laws, there is  a close relation of the committors and the effective current (when the chain is reversible) to the voltage and the electric current in an electric resistor network \cite{doyle1984random, norris1998markov} with a voltage applied between $A$ and $B$, see \cite{metzner2008transition, metzner2009transition} for work in this direction.
Also in \cite{metzner2008transition, metzner2009transition},  a decomposition algorithm of the effective current into the paths from $A$ to $B$ carrying a substantial portion of the transition rate is proposed, yielding the dominant transition channels between $A$ and $B$.
The effective current of reactive trajectories is loop-erased, therefore in \cite{banisch2015reactive} a decomposition of the current of reactive trajectories into cyclic structures and non-cyclic parts  is suggested.  
}\label{rem:network_flows}

Further, since
 $\smash{ \sum_{i\in A, j\in\St} \current_{ij} = \sum_{ i\in\St, j\in B} \current_{ij} }$ by Theorem \ref{Thm:conservationlaw},  the discrete rate of leaving $A$ equals the discrete rate of entering into $B$, thus we denote $\rate :=  \rateA = \rateB $. This tells 
  us the probability of a realized transition per time step, i.e., either the probability to leave $A$ and be on the way to $B$ next, or the probability to reach $B$ when coming from $A$. That $\rate$ indeed has the physical interpretation of a rate becomes clear from the characterization in Theorem~\ref{thm_traj_average} below.

\subsection{Interpretation of the Statistics as Time-averages along Trajectories} 
The statistics from the previous section give us dynamical information about the ensemble of reactive trajectories.
Due to the ergodicity, the Markov chain will visit all states infinitely many times and  the ensemble space average of a quantity  equals the time average this quantity along  a single infinitely long trajectory (Birkhoff's ergodic theorem).  Therefore the reactive trajectory statistics from the previous section can also be found by considering the reactive pieces along a single infinitely long trajectory and by averaging over them.
  \begin{Thm} \label{thm_traj_average} For a  Markov chain $(X_n)_{n\in\Z}$ satisfying Assumption \ref{ass:ergodic}, we have the following $\prob-$almost sure convergence results:
    \begin{equation}  \begin{split}
        \reacdist_i &= \lim_{N \to \infty} \frac{1}{2N+1} \sum_{n=-N}^N \1_{\{i\}} (X_n) \1_A \left(X_{\tau_{A\cup B}^-(n)}\right) \1_B \left(X_{\tau_{A\cup B}^+(n)}\right) \\
        \current_{ij} &=\lim_{N \to \infty} \frac{1}{2N+1} \sum_{n=-N}^N \1_{\{i\}} (X_n) \1_A \left(X_{\tau_{A\cup B}^-(n)}\right) \1_{\{j\}} (X_{n+1}) \1_B \left(X_{\tau_{A\cup B}^+(n+1)}\right)  \\
        \rate &= \lim_{N \to \infty} \frac{1}{2N+1} \sum_{n=-N}^N
         \1_{A} (X_n) \1_B\left(X_{\tau_{A\cup B}^+(n+1)}\right) = \lim_{N \to \infty} \frac{1}{2N+1} \sum_{n=-N}^N
         \1_A\left(X_{\tau_{A\cup B}^-(n-1)}\right) \1_{B} (X_n) 
    \end{split} \end{equation}
    where $i,j \in \St$ and $\1_A(x)$ is the indicator function on the set $A$.
    \end{Thm}
    The proof of Theorem \ref{thm_traj_average} can be found in \cite[Thm 3.3.2, Thm 3.3.7, Thm 3.3.11]{RiberaBorrell2019} and relies on   Birkhoff's ergodic theorem \cite{walters2000introduction} for the canonical representation of the process as a Markov shift.
    
    The theorem not only offers a data-driven  approach to approximate the transition statistics by averaging along a given sufficiently long trajectory sample but also gives interpretability to the statistics. While $\reacdist$ as the (not normalized) trajectory-wise distribution of reactive trajectories and $ \current$ as the trajectory-wise flux of reactive trajectories are still straightforward to understand, we can also give meaning to the rate  $\rate$ and to $Z^{AB}=\sum_{i \in C} \reacdist_i$. We can think of $\rate$ as the total number of reactive transitions taking place within the time interval $\{-N,\dots,N\}$ divided by the number of time steps $2N+1$ in the limit of $N\to \infty$.
Similarly we can give meaning to $Z^{AB}$ as the total time spent transitioning during $\{-N,\dots,N\}$ divided by $2N+1$ in the limit of $N\to \infty$. 
Last, we note that the ratio between $ Z^{AB}$ and $\rate$  provides us with a further characteristic transition quantity~\cite{vanden2006transition},
\begin{equation*} \label{eq:mean_reactive_time}
\meantime := \frac{Z^{AB}}{k^{AB}} = \lim_{N \to \infty} \frac{ \text{total time spent transitioning during }\{-N,\dots,N\}}{ \#\{\text{reactive transitions  during }\{-N,\dots,N\}\} },
\end{equation*}
telling us the total time spent transitioning divided by the total number of transitions, i.e., the \emph{expected length of a transition} from $A$ to $B$.

\subsection{Numerical Example 1: Infinite-time, Stationary Dynamics on a 5-state Network}
\label{sub:network_ergodic}

We consider a Markov chain on a   network of five states $\St = \{0,1,2,3,4\}$ of which  $0,2,4$ have a high probability to remain in the same state. The transition matrix is given by the row-stochastic matrix $\trans$, see Figure \ref{fig:example_inf}. We are interested in transitions from the subset $A = \{0\}$ to the subset $B=\{4\}$. What is the most likely route that the transitions take?

The numerically computed committors and transition statistics are shown in Figure \ref{fig:example_inf}, the discrete rate of transitions is $\rate = 0.018$, i.e., on average every 56 time steps a transition from $A$ to $B$ is completed. We see that the effective current is the strongest along the path $A \rightarrow 1 \rightarrow B$, i.e., most transitions effectively happen via this path, but a share of $33,3\%$ of  transitions happen via the route from $A \rightarrow 3 \rightarrow B$.  Also, we can note that the density of reactive trajectories has a very high value in the state $2$, indicating that many reactive trajectories pass this state or stay there for a long time. Thus reactive trajectories on the effective path  $A \rightarrow 1 \rightarrow B$ will  likely do a detour to $2$, which is not visible from the effective current since it does not tell us about detours.  \\ 

\begin{figure}
\begin{subfigure}[t]{0.20\textwidth}
\centering
\includegraphics[width=1\textwidth]{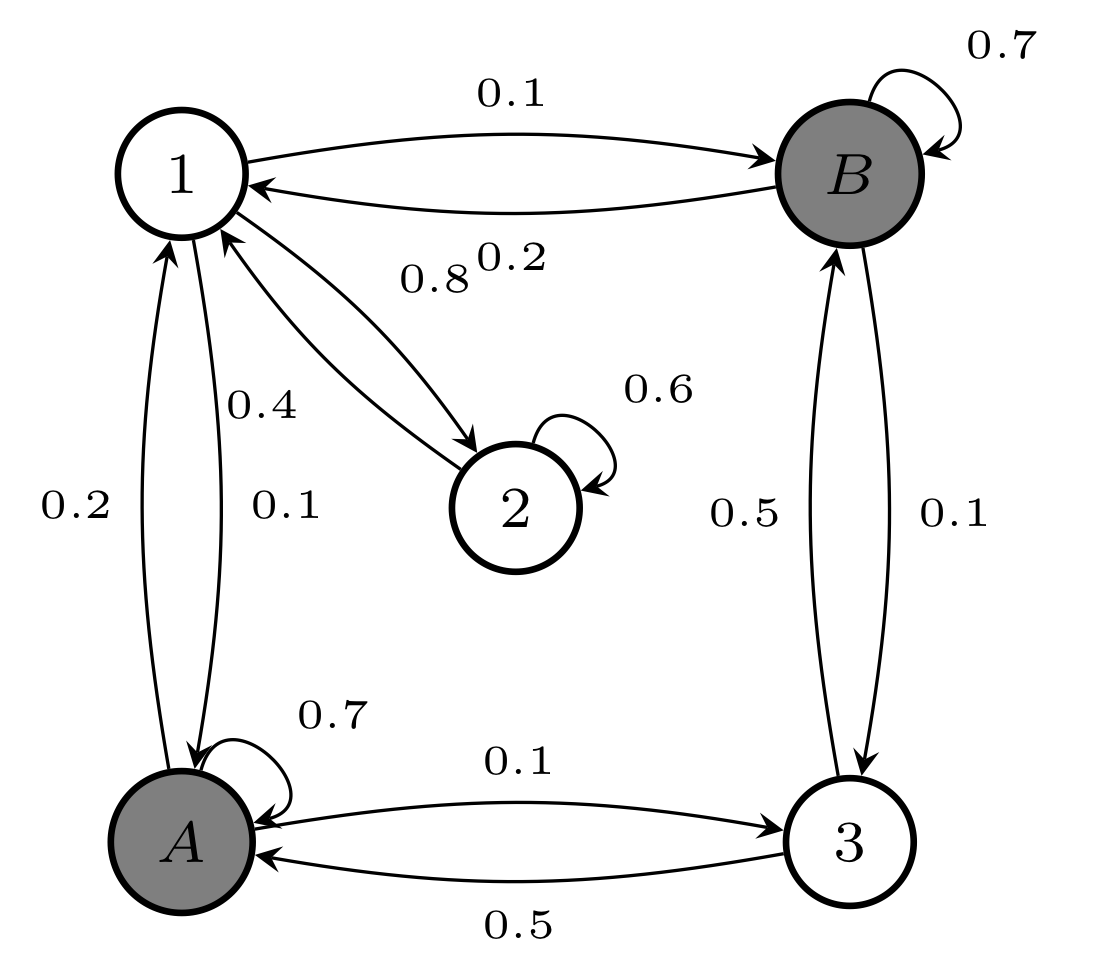}
\caption{}
\end{subfigure}
\begin{subfigure}[t]{0.18\textwidth}
\centering
\includegraphics[width=1\textwidth]{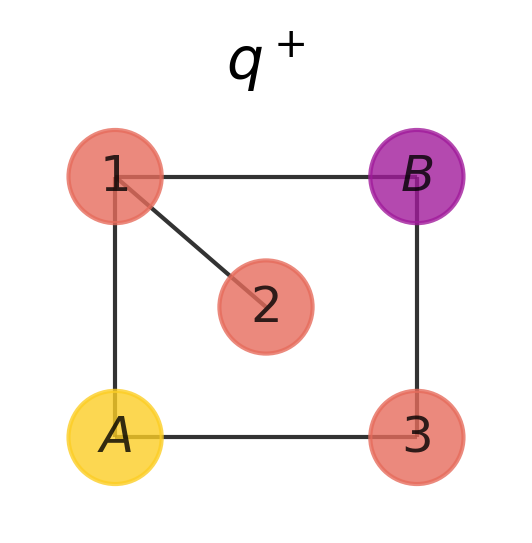}
\caption{}
\end{subfigure}
\begin{subfigure}[t]{0.18\textwidth}
\centering
\includegraphics[width=1\textwidth]{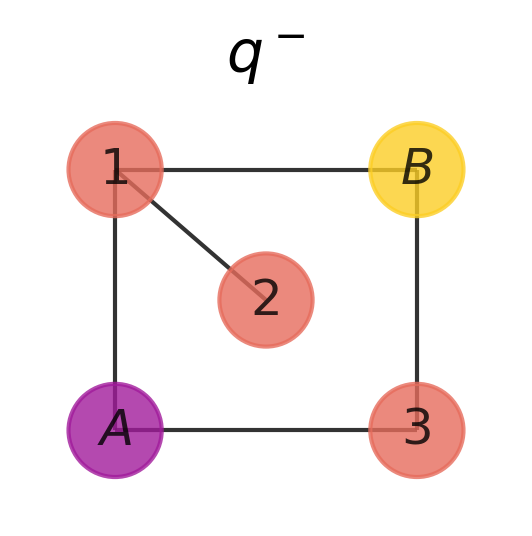}
\caption{}
\end{subfigure}
\begin{subfigure}[t]{0.18\textwidth}
\centering
\includegraphics[width=1\textwidth]{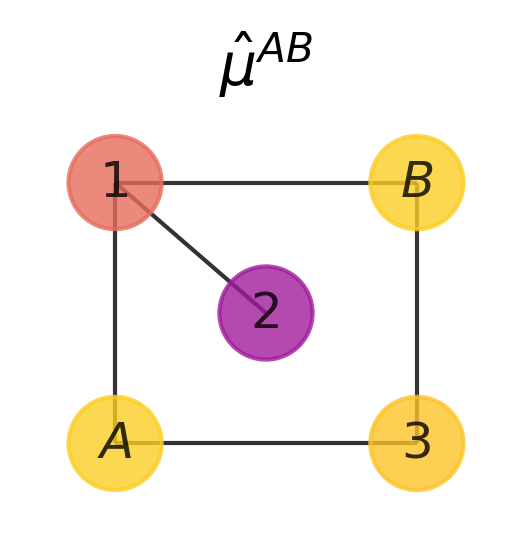}
\caption{}
\end{subfigure}
\begin{subfigure}[t]{0.18\textwidth}
\centering
\includegraphics[width=1\textwidth]{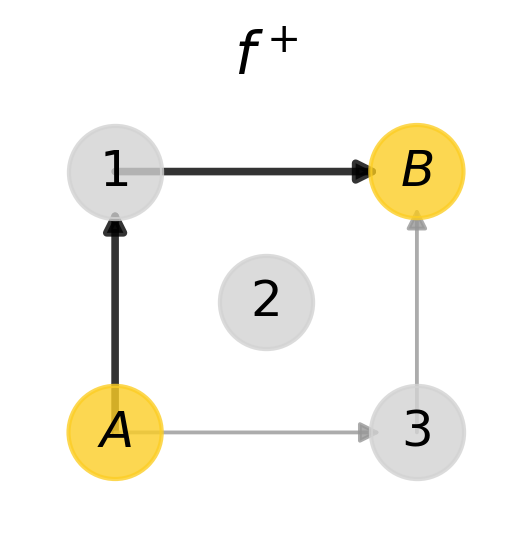}
\caption{}
\end{subfigure}
\begin{subfigure}[t]{0.03\textwidth}
\centering
\includegraphics[width=1\textwidth]{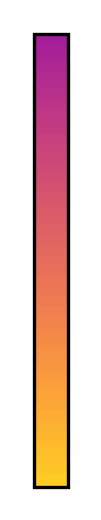}
 
\end{subfigure}
\caption{The dynamics of the infinite-time system in \hyperref[sub:network_ergodic]{Example 1} with \textbf{(a)} transition matrix $\trans$ and statistics of the transitions from $A$ to $B$ are shown. The node color of  \textbf{(b)} $q^+$,  \textbf{(c)} $q^-$ and  \textbf{(d)} $\reacdistnorm$ indicate the values of the respective statistics, ranging from yellow (low values) via orange to purple (high value).  \textbf{(e)} The strength of the effective current from one state $i$ to state $j$ is shown by the thickness and greyscale of the arrow pointing from $i$ to $j$. The node color of $A$ and $B$ indicates the sum of current flowing out respectively into the node, here this is just $\rate$.} \label{fig:example_inf}
\end{figure}

\section{TPT for Periodically Driven, Infinite-time Markov Chains} \label{sec:tpt_periodic}
 Many real-world systems showcase periodicity, for example any system subject to seasonal driving, or physical systems with periodic external stimuli.
 
For studying transitions in  these systems, we extend TPT to Markov chains with periodically varying transition probabilities  that are equilibrated to the forcing and cycle through  the same distributions each period. If the period is only one time step long, this case reduces to the  previous case of stationary, infinite-time dynamics.

We start by laying out the exact setting of the process that we consider, before turning to the computation of committors and transition statistics for periodically forced dynamics. As we will see, by writing the committor equations on a time-augmented state space, we can also find committors for systems with stochastic switching between different dynamical regimes.
 
\subsection{Setting} 
\Ass{Consider a Markov chain $(X_n)_{n\in\Z}$  on a finite and discrete state space $\St$ with transition probabilities $\trans_{ij}  (n) =   \prob(X_{n+1}=j \, | \, X_n=i)  $ that are periodically varying in time with period length $M\in \N$, i.e., the transition matrices fulfill $$ \trans (n) = \trans (n+M)\ \forall n\in \Z.$$}\label{ass:periodic_matrices}
Therefore  the transition matrices at the times within one period $\M := \{0,1,\dots,M-1\}$ are sufficient to describe all the dynamics and we denote them by 
$\trans_m := \trans(n)$ for time $n\in\Z$  congruent to $m\in\M$ modulo $M$.

The product of transition matrices over one $M-$period starting at a time equivalent to $m\in\M$ (modulo $M$) is denoted by  $\bar{P}_m := \trans_m \trans_{m+1} \cdots \trans_{m+M-1}$, which is again a transition matrix pushing the Markov chain $M$ time instances in time forward starting from time~$\equiv m \Mod{M}$, see  Figure \ref{fig:periodic_setting}. The chain described by $\bar{\trans}_m$ is not time-dependent anymore, it resolves the state of the original system only every $M$ time instance with a time-independent transition matrix $\bar{\trans}_m$.
    \begin{Prop}\label{thm:periodic_dist}
    If $\bar{\trans}_0$ is irreducible and assuming the setting \ref{ass:periodic_matrices} as described above, then for all $m\in \M$ there exists an  
    invariant distribution  $\pi_m$ such that $ \pi_m^\top  = \pi_m^\top \bar{\trans}_m $ of the transition matrix $\bar{\trans}_m$.
    Further, $\pi_0$ is unique and $\pi_{0,i}>0$ for all $i\in\St$. If we in addition require $ \pi_{m+1}^\top  = \pi_m^\top P_m$ for all $m$, then the entire family $(\pi_m)_{m=0,\ldots,M-1}$ is unique.\footnote{Sometimes such a family of invariant densities is called \emph{equivariant}.}
    \end{Prop}
    \begin{proof}
    Since $\bar{\trans}_0$ is irreducible and the state space is finite, the Markov chain induced by  $\bar{\trans}_0$  has a unique and positive  invariant  density $\pi_0=(\pi_{0,i})_{i\in\St}$ such that $\pi_0^\top \bar{\trans}_0 = \pi_0^\top$.
    It follows that also $\bar{\trans}_1$ has an invariant density, namely $\pi_1^\top:=\pi_0^\top \trans_0 $,  since it fulfills
    \begin{equation} 
        \pi_1^\top \bar{\trans}_1 = \pi_0^\top \trans_0  \bar{\trans}_1 =  \pi_0^\top \trans_0  \trans_1  \cdots \trans_{M-1}  \trans_{0}=  \pi_0^\top \bar{\trans}_0 \trans_{0} =  \pi_0^\top    \trans_{0}  =  \pi_1^\top.  \nonumber
    \end{equation}
    In the same way,  ${\bar{\trans}_m, m=2,\dots, M-1,}$ have an  invariant  distribution ${\pi_m^\top :=\pi_{m-1}^\top \trans_{m-1}}$, such that ${\pi_m^\top \bar{\trans}_m = \pi_m^\top}$.
    \end{proof} 

Thus, the densities $\pi_1,\dots \pi_{M-1}$ are not necessarily unique, unless we require irreducibility. Doing so, there is a unique periodic family of distributions that the chain can admit, and we call such a chain $M-$stationary; see Figure \ref{fig:periodic_setting}. Having the long-time behavior of chains in mind in this section, we will assume this property, relying on ergodicity.
    \begin{SCfigure}[1][!h]
    \centering
    \includegraphics[width=0.55\textwidth]{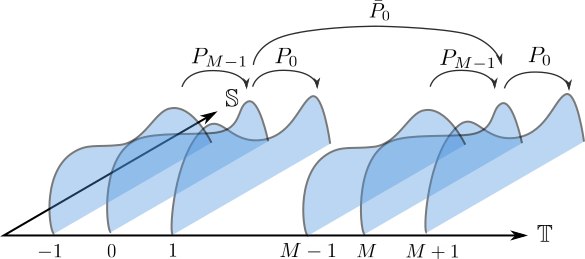}
    \caption{The evolution of densities for periodic dynamics assuming $M-$stationarity. The density evolution described by $\trans_m$, $ m\in\M$ varies periodically with time $n$. Whereas the dynamics described by $\bar{\trans}_m$ for a fixed $m$ are time-independent, they resolve the system only every $M$ time instances at times $n$ equivalent to $m$ modulo $M$.}
    \label{fig:periodic_setting}
\end{SCfigure}
\Ass{We   assume that $\bar{\trans}_0$ is irreducible and that the chain is $M$-stationary, i.e., $\prob(X_n=i)=\pi_{m,i}$ whenever $n$ is equivalent to $m$ modulo $M$.}\label{ass:periodic_densities}

Next, we   introduce the time-reversed chain $(\back{X}_n)_{n\in\Z}$ with $\back{X}_n = X_{-n}$. Due to $M-$stationarity the transition probabilities of the time-reversed chain are also $M-$periodic, and it is enough to give the transition probabilities backward in time $\transback_{m}$ for each time point during the period $m\in \M$
    \begin{equation} \label{eq:periodic_backward} \begin{split}
       \transback_{m,ij} :&=\prob(\back{X}_{M-m+1}=j|\back{X}_{M-m}=i) =  \prob(X_{m-1}=j|X_m=i) \nonumber \\
        &= \prob(X_m=i|X_{m-1}=j) \frac{\prob(X_{m-1}=j)}{\prob(X_{m}=i)} = \trans_{m-1,ji} \frac{\pi_{m-1,j}}{\pi_{m,i}} \nonumber
        \end{split} \end{equation}
        whenever $\pi_{m,i}>0$, else for $\pi_{m,i}=0$ we set $\transback_{m,ij} :=0$.

\subsection{Committor Probabilities}\label{sec:com_periodic}
We will first look at the forward and backward committors and the system of equations that can be solved to acquire them. The forward $q^+(n)$ respectively backward committor $q^-(n)$ is defined as before in  \eqref{eq:q_f_def} respectively \eqref{eq:q_b_def}, but now since the law of the Markov chain is the same every $M$ time instances, the committors also vary periodically and are identical every $M$ time steps, we therefore denote \begin{equation}\label{eq:q_periodic_def}
    \begin{split} 
    q^+_{m}&:= q^+(n) \\
    q^-_{m}&:= q^-(n) 
    \end{split}
    \end{equation}
    whenever  $n$ is equivalent to  $m\in\M$.\\
    Again, we consider non-empty and disjoint subsets $A$, $B$ of the state space.  It is  straightforward to extend the theory to  periodically varying sets $A = (A_m)_{m\in\M}$, $B=(B_m)_{m\in\M}$ defined on~$\St^M$.
    \begin{Thm} \label{Thm:q_fb_periodic}
    Suppose a Markov chain  meeting the Assumptions \ref{ass:periodic_matrices} and \ref{ass:periodic_densities}.
    Then the $M-$periodic forward committor $q^+_{m}=(q^+_{m,i})_{i\in\St}$  fulfills the following iterative  system with periodic conditions  $q_{M}^+=q_{0}^+$
\begin{equation}\label{eq:q_f_periodic}
\left\{ \begin{array}{rcll}
q_{m,i}^+ &=& \sum\limits_{j\in \St} \, \trans_{m,ij} \, q_{m+1,j}^+ & i \in C \\
    q_{m,i}^+ &=& 0& i \in A  \\
    q_{m,i}^+ &=& 1& i \in B  \\
\end{array}\right.
\end{equation}
 Whereas the $M-$periodic backward committor $q^-_{m}=(q^-_{m,i})_{i\in\St}$ satisfies
\begin{equation}\label{eq:q_b_periodic}
\left\{ \begin{array}{rcll}
q_{m,i}^- &=& \sum\limits_{j\in \St} \, \transback_{m,ij} \, q_{m-1,j}^- & i\in C \\
    q_{m,i}^- &=& 0& i \in B \\
    q_{m,i}^- &=& 1& i \in A  \\
\end{array}\right.
\end{equation}
where $q_{M}^-=q_{0}^-$.
    \end{Thm}
    The proof follows the lines of the proof above for the stationary, infinite-time case  and can be found in the Appendix \ref{app:proofs}.
    
    Before proving that the two systems are uniquely solvable, we characterize the forward and backward committors in terms of the path probabilities over one period $\M$:
    \begin{Lemma}\label{Lm:stacked_periodic}
    For any time $n=m$ modulo $M$ and $i \in C$ the committor functions \eqref{eq:q_periodic_def} satisfy the following equalities
        \begin{equation} \label{eq:q_f_stacked_m} 
    \begin{split}
     q_{m,i}^+  &= \sum_{i_1 \dots i_{M} \in C} \trans_{m,ii_1}  \trans_{m+1,i_1i_2} \cdots \trans_{m+M-1,i_{M-1} i_M}  q_{m,i_M}^+ + \sum_{\tau=1}^{M} \sum_{\substack{i_1\dots i_{\tau-1} \in C \\ i_{\tau}\in B}} \trans_{m,i i_1}  \cdots \trans_{m+\tau-1,i_{\tau-1}i_\tau}
    \end{split}
    \end{equation}
    \begin{equation} \label{eq:q_b_stacked_m}
    \begin{split}
     q_{m,i}^-  &= \sum_{i_{1} \dots i_{M} \in C} \transback_{m,ii_{1}}  \transback_{m-1,i_{1}i_{2}} \cdots \transback_{m-M+1,i_{M-1} i_{M}}  q_{m,i_{M}}^- + \sum_{\tau=1}^{M} \sum_{\substack{i_1\dots i_{\tau-1} \in C \\ i_{\tau}\in A}} \transback_{m,i i_1}  \cdots \transback_{m-\tau+1,i_{\tau-1}i_\tau}   
    \end{split}
    \end{equation}
    \end{Lemma}
    \begin{proof}
    First it follows from \eqref{eq:q_f_periodic} for $i\in C$    that
\begin{equation}
    q_{m,i}^+  = \sum_{i_1\in C} \trans_{m,ii_1} q_{m+1,i_1}^+ + \sum_{i_1\in B} \trans_{m,i i_1} 
\end{equation}
since $q^+_{m+1,i_1} = 1$ if $ i_1\in B $, $q^+_{m+1,i_1} = 0$ if $ i_1\in A $. By inserting the committor equations at the following times iteratively and by using that $q^+_0=q^+_M$, we get \eqref{eq:q_f_stacked_m}. We can proceed analogously for the backward committor, starting from \eqref{eq:q_b_periodic}   and re-inserting committor equations.  
\end{proof}
The equation \eqref{eq:q_f_stacked_m} with $m=0$ only contains one unknown and can be solved, e.g., numerically for all $i\in C$, whereas for $A$ respectively $ B$ the committor is simply $0$ respectively $1$. The committor for the remaining times $m=1,2,..$ can then be computed thereof by using \eqref{eq:q_f_periodic}. Analogously for \eqref{eq:q_b_stacked_m}.

The time-resolution of the Markov chain during the period is important for the committors since we can resolve hitting events of $B$ within the period. The committors one would compute for a more coarsely resolved chain without state information during the period, i.e., for the chain described by $\bar{P}_0$ (time-homogeneous, but mapping one period in time forward),  will not notice that the chain has hit $B$ at times other than $m=0$. To see that, compare Lemma \ref{Lm:stacked_periodic}  with   Lemma \ref{Lm:q_fb_all_paths} using $\bar{P}_0$.
\begin{Lemma} \label{Lm:ex_uniq_periodic}
By the irreducibility of $\bar{\trans}_0$, the solutions to \eqref{eq:q_f_periodic} and \eqref{eq:q_b_periodic} exist and are unique.
\end{Lemma}
The proof  can be found in the Appendix \ref{app:proofs}.
\remark{The committor equations can also be written on a time-augmented state space using a period-augmented transition matrix that pushes the dynamics deterministically forward in time
    $$ \trans_\text{Aug} =
\begin{pmatrix}
0 &  \trans_0 &  & 0 \\
  &  \ddots & \ddots & \\
 &   &  0&\trans_{M-2}  \\
\trans_{M-1} & &   & 0
\end{pmatrix}.
$$ 
Extending this approach, one can also consider committor equations for systems that switch stochastically with probabilities $\hat{P}\in\R^{M\times M}$ between $M$ different regimes, each regime is described by a transition matrix $P_m$. This is essentially the Markov-chain analogue of a \emph{random dynamical system}~\cite{arnold1995random}.
The regime-augmented transition matrix is given by
    $$ \trans^{\text{switch}}_\text{Aug} =
\begin{pmatrix}
\hat{P}_{11} P_1 &  \dots   & \hat{P}_{1M} P_M \\
\vdots & \ddots &  \vdots\\
\hat{P}_{M1} P_1&  \dots  & \hat{P}_{MM} P_M
\end{pmatrix}.
$$ 
We refer to the Appendix \ref{app:augment_periodic} for more details on both ansatzes and the computation of committor probabilities on the augmented space. 
The augmented approach also offers a numerical way of solving the committor equations with periodic boundary conditions.}\label{rem:periodic_augmented}

\subsection{Transition Statistics}

We have seen that the forward and backward committor in the case of periodically driven dynamics are also $M-$periodic and can be computed from the iterative equations \eqref{eq:q_f_periodic}, \eqref{eq:q_b_periodic} with periodic conditions in time. Since committors, densities and transition matrices are $M-$periodic, all statistics computed thereof are so too by the theory in Section~\ref{sec:statistics}, and we equip them with a subscript $m$, e.g., $\reacdistnorm(n) = \reacdistnorm_m$, $\current(n) = \current_{m} $, whenever $n\equiv m$ modulo~$M$.

Compared to the previous case of stationary, infinite-time, the  discrete rate of reactive trajectories leaving $A$ at time $m$,  $\smash{ \rateA_m = \sum_{i\in A, j\in\St} \current_{ij}(m) }$, does not anymore equal the discrete rate of reactive trajectories  arriving in $B$ at time $m$, $\smash{ \rateB_m = \sum_{i\in \St, j\in B} \current_{ij}(m-1) }$. 

The next theorem provides us with the reactive current conservation laws in the case of periodic dynamics and will allow us to find the relation between $\rateA_m$ and $\rateB_m$.
 \begin{Thm} \label{Thm:conservationlaw_periodic} 
 Consider a Markov chain satisfying Assumptions  \ref{ass:periodic_matrices} and \ref{ass:periodic_densities}. Then,  for each node $i\in C$ and time $m\in\M$ we have the following
 current conservation law 
\begin{align}
\sum_{j\in \St} &   f_{m,ij}^{AB} = \sum_{j\in \St} f_{m-1,ji}^{AB},
\end{align} 
i.e., all the reactive trajectories that flow out of $i$ at time (congruent to) $m$, have flown into $i$ at time equivalent to $m-1$. 

Further, over one period the amount of reactive flux leaving $A$ is the same as the amount of flux entering $B$, i.e., 
\begin{align} \label{eq:flux_AB_m}
    \sum_{m\in\M} \sum\limits_{\substack{i \in A\\ j\in \St}} f_{m,ij}^{AB} =  \sum_{m\in\M}  \sum\limits_{\substack{i\in \St\\ j\in B}} f_{m,ij}^{AB}.
\end{align}
\end{Thm}
The proof can be found in the Appendix \ref{app:proofs} and follows from straightforward computations using the committor equations to rewrite the reactive current.

As a result of \eqref{eq:flux_AB_m}, the discrete out-rate averaged  over one period equals the average discrete in-rate, which we define to be $\bar{k}^{AB}_{M}$, 
i.e., 
$$  \bar{k}^{AB}_{M}:= \frac{1}{M} \sum_{m\in\M} \rateA_m  =  \frac{1}{M} \sum_{m\in\M} \rateB_m .$$ 
This period-averaged discrete rate tells us the average probability per time step of a reactive trajectory to depart in $A$ or in other words, the expected number of reactive trajectories leaving $A$ per time step. 
\subsection{Numerical Example 2: Periodically Varying Dynamics  on a 5-state Network}
\label{sub:network_periodic}

\begin{figure}
\centering
\begin{subfigure}[c]{0.3\textwidth}
\centering
\includegraphics[width=1\textwidth]{img/tikz_T_plus_L.png}
\caption{}
\end{subfigure}
\begin{subfigure}[c]{0.3\textwidth}
\centering
\includegraphics[width=1\textwidth]{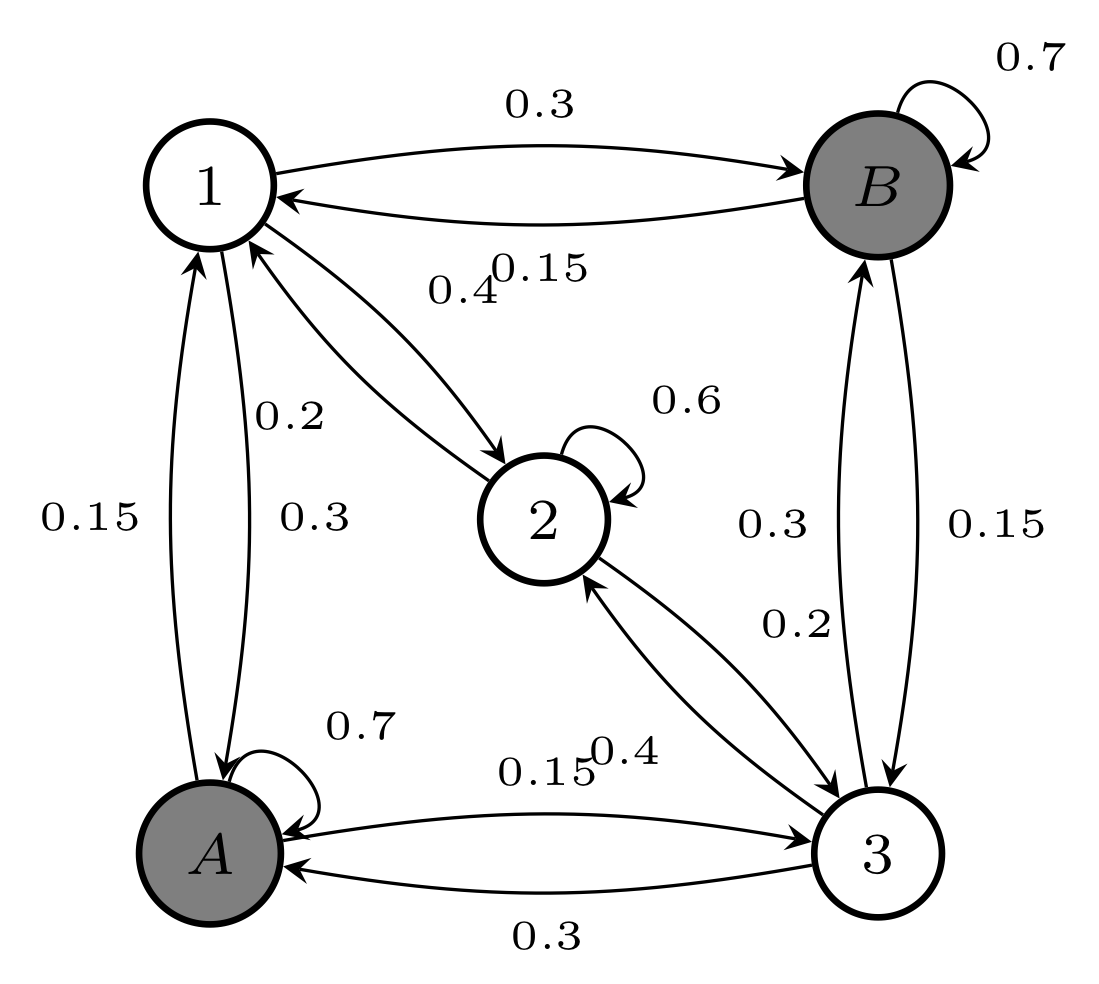}
\caption{}
\end{subfigure}
\begin{subfigure}[c]{0.3\textwidth}
\centering
\includegraphics[width=1\textwidth]{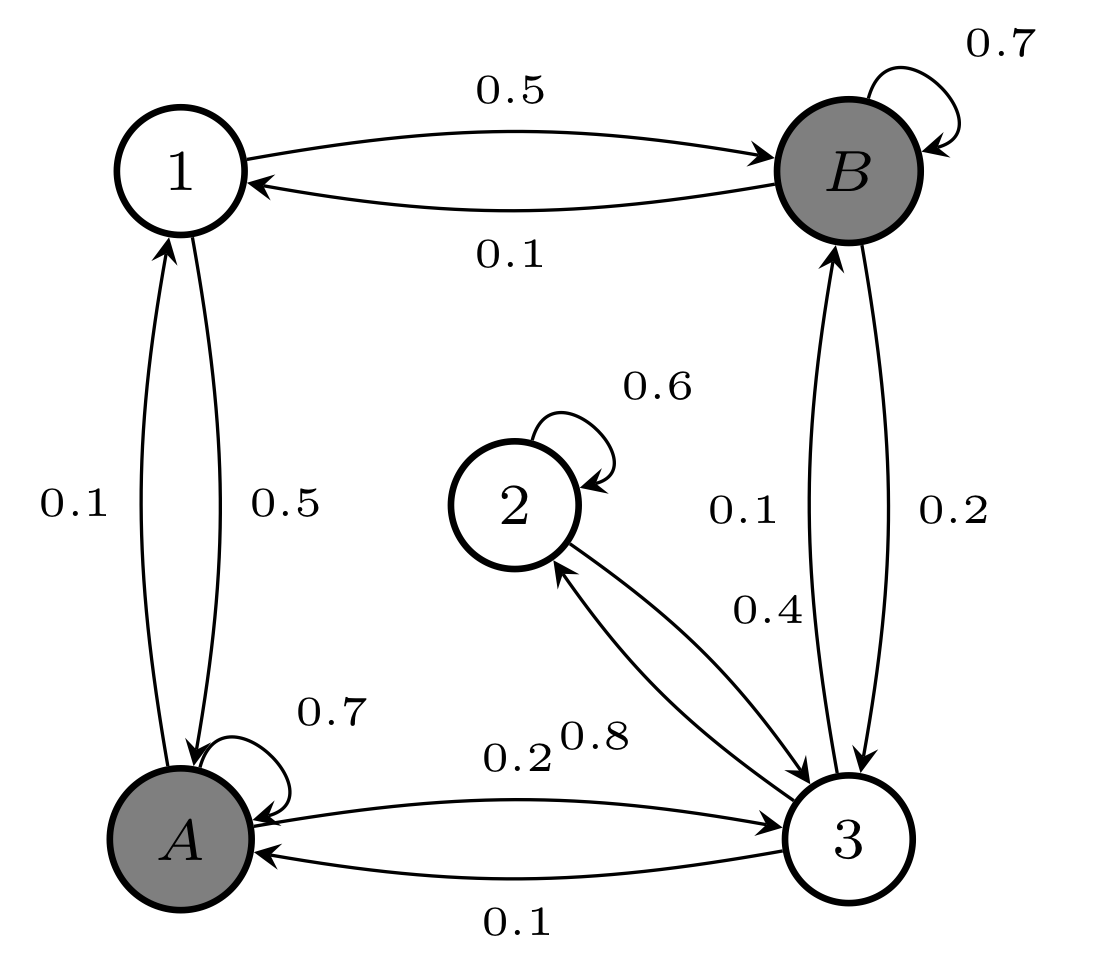}
\caption{}
\end{subfigure}
\caption{The dynamics of the periodic system  of \hyperref[sub:network_periodic]{Example 2} are given by the transition matrix $P_m = T + \cos\left(\frac{m\pi}{3}\right)L$ for the times $m = n \, (\text{mod 6})$;  shown here are the transition matrices \textbf{(a)} $T+L$, \textbf{(b)} $T$ and \textbf{(c)} $T-L$. }\label{fig:example_p_trans}
\end{figure}

\begin{figure}
\centering
\begin{subfigure}[t]{0.96\textwidth}
\centering
\includegraphics[width=1\textwidth]{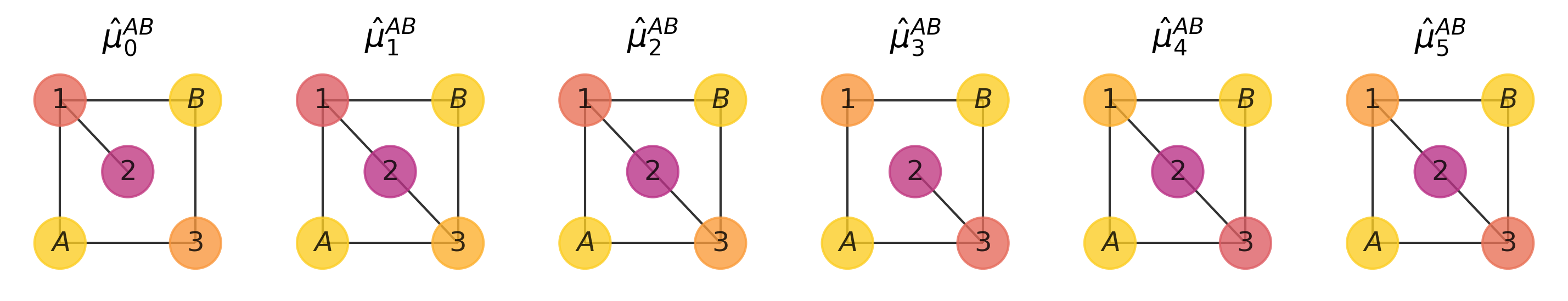}
\caption{}
\end{subfigure}
\begin{subfigure}[t]{0.03\textwidth}
\centering
\includegraphics[width=1\textwidth]{img/small_network_colorbar.png}
\end{subfigure}
\begin{subfigure}[t]{0.96\textwidth}
\centering
\includegraphics[width=1\textwidth]{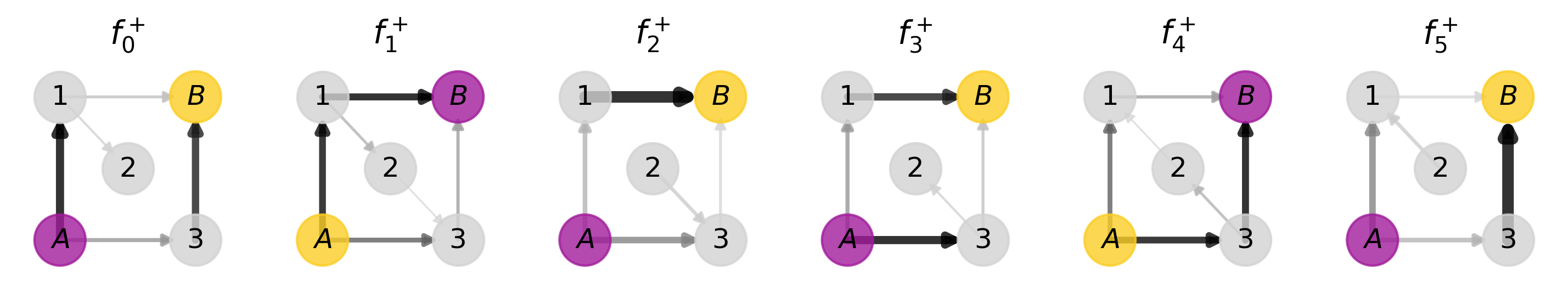}
\caption{}
\end{subfigure}
\begin{subfigure}[t]{0.03\textwidth}
\centering
\includegraphics[width=1\textwidth]{img/small_network_colorbar.png}
\end{subfigure}

\caption{We show statistics of \hyperref[sub:network_periodic]{Example 2}, namely,  \textbf{(a)} the  normalized distribution which  is shown by the color of the nodes, ranging from yellow (low values) via orange to purple (high value), as well as  \textbf{(b)} the  effective current from one state $i$ to state $j$ that is indicated by the thickness and greyscale of the arrow pointing from $i$ to $j$.  The node color of the effective current plot indicates the outflow of $A$ and the inflow into $B$ of current at the corresponding time point $m$, which can be related to the $\rateA_m$ and $\rateB_{m+1}$.}
\label{fig:example_p}
\end{figure} 

 We consider the same $5-$state network as before in \hyperref[sub:network_ergodic]{Example 1}, but this time the Markov chain is in equilibrium to the periodically varying  transition probabilities
\[
\trans_m = T + \cos\left(\frac{2m \pi}{M}\right)\, L,\ m\in\{0,\dots,M-1\},
\]
with period length~$M=6$. The transition matrices are chosen such that for time $m=0$ the dynamics are the same as in the stationary Example \hyperref[sub:network_ergodic]{1}, $P_0=T+L=P$, whereas at time $m=3$ the dynamics are reversed $P_3 = T-L$. Matrix  $T$ is a fixed row-stochastic matrix which has transition probabilities that are symmetric along the axis through $A$ and $B$ and $L$ is a 0-rowsum matrix that is not symmetric along the axis through $A$ and~$B$. The probabilities of the transition matrices are shown in Figure \ref{fig:example_p_trans}.

With the numerically computed transition statistics as shown in Figure \ref{fig:example_p} we try to answer whether the added perturbation results in alternative effective paths compared to \hyperref[sub:network_ergodic]{Example 1}.  
The effective current indicates that the most likely effective transition path from $A$ to $B$   either goes via $A \rightarrow 1\rightarrow B$ (with a likely detour to $2$) during the first half of the period, or via $A\rightarrow 3\rightarrow B$ (with a detour to $2$) towards the second half   of the period. But interestingly, we have additional transition paths that go via $A\rightarrow 1\rightarrow 2\rightarrow 3\rightarrow B$ and~$A\rightarrow 3\rightarrow 2\rightarrow 1\rightarrow B$.   Neither in the stationary system described by $P_0=T+L$ nor by $P_3 = T-L$ this path would be possible. Additionally, the period-averaged rate $\bar{k}^{AB}_{6} = 0.034$ is higher than the rate in the stationary case (\hyperref[sub:network_ergodic]{Example 1}).

\section{TPT for Markov Chains on a Finite Time Interval} \label{sec:tpt_finite}
We now develop TPT for  Markov chains with the transitions of interest taking place during a finite time interval. The transition rules can be time-dependent and the dynamics can be non-stationary, i.e., out of equilibrium. 
 
The resulting committor equations (Section \ref{sec:com_finite}) and statistics (Section \ref{sec:stat_finite}) in case of finite-time dynamics are similiar as in the periodic case, yet there are some  distinctions. The committor equations are now equipped with final respectively initial conditions and the statistics show some boundary effects at the  time interval limits.

In Section \ref{sec:convergence_finite} we also provide a consistency result between the stationary, infinite-time case and the  finite-time case for a stationary Markov chain, by considering the limit of the time-interval going to infinity. 

\subsection{Setting}
Let us start by describing the systems of interest  in this section.
\Ass{We consider a Markov chain on a finite time interval $(X_n)_{0 \leq n \leq N-1}$, $N \in \N$, taking values in a discrete and finite space $\St$.
The probability of a jump from the state $i \in \St$ to the state $j \in \St$ at time $n \in \{0, \dots, N-2\}$ is given by the $(i,j)-$entry of the row-stochastic transition matrix $\trans(n) = (\trans_{ij}(n))_{i,j \in \St}$:
\[ P_{ij}(n) := \prob(X_{n+1} = j \,|\, X_n = i) . \] 
Setting the initial density $\lambda(0) = (\lambda_i(0))_{i \in \St}$, the densities at later times $n\in\{1,\dots,N-1\}$ are given by $\lambda(n+1)^\top = \lambda(n)^\top P(n) $. 
}\label{ass:finite_chain}

By these assumptions, the chain can have time-inhomogeneous transition probabilities, or even if $P(n)=P$ for all $n$, the densities $\lambda(n)$ can be changing in time. Also, we are not requiring the chain to be irreducible anymore.

The time-reversed process $(\back{X}_n)_{0 \leq n \leq N-1}$ defined by ${\back{X}_n := X_{N-1-n}}$ is also a Markov chain \cite[Thm 2.1.18]{RiberaBorrell2019}. Its transition probabilities are given for any $n \in \{1, \dots, N-1\}$ by
\begin{equation} \label{eq:transback_finite}
\begin{split} 
\transback_{ij}(n)
&:= \prob(\back{X}_{N-1-n+1} = j \,|\, \back{X}_{N-1-n} = i) \\
&= \prob(X_{n-1} = j \,|\, X_{n} = i) = \frac{\dist{n-1}{j}}{\dist{n}{i}} \, \trans_{ji}(n-1)
\end{split}
\end{equation}
whenever $ \dist{i}{n} >0$.
From the backward transition probabilities \eqref{eq:transback_finite}, we note that even the time-reversed process of a finite-time, time-homogeneous Markov chain (i.e., $P(n)=P$ for all $n$) is in general a finite-time, time-inhomogeneous Markov chain, unless also the distribution $\dist{n}{}$ is time-independent.

\subsection{Committor Probabilities}\label{sec:com_finite}
The forward \eqref{eq:q_f_def} and backward committors \eqref{eq:q_b_def} keep their dependence on the time of the chain $n \in \{ 0, \dots, N-1\}$. The following theorem provides us with two iterative equations for the forward and backward committor. 
Because one can solve \eqref{eq:q_f_finite} and~\eqref{eq:q_b_finite} iteratively, the solutions  exist and are unique.
\begin{Thm} \label{Thm:q_fb_finite}
The forward committor for a finite-time Markov chain of the form \ref{ass:finite_chain}  satisfies the following iterative system of equations for $n \in \{0, \dots, N-2\}$:
\begin{equation} \label{eq:q_f_finite}
\left\{ \begin{array}{rcll}
q_i^+(n) &=&  
\sum\limits_{j \in \St} \, \trans_{ij}(n) \, q_j^+(n+1)    & i \in C \\
q_i^+(n) &=& 0 & i \in A \\
q_i^+(n) &=& 1 & i \in B
\end{array}\right.
\end{equation}
with final condition $q^+_i(N-1) = \1_B(i)$.
Analogously, the backward committor for a finite-time Markov chain  satisfies for ${n \in \{1, \dots, N-1\}}$
\begin{equation} \label{eq:q_b_finite}
\left\{\begin{array}{rcll}
q_i^-(n) &=&  
\sum\limits_{j \in \St} \, \transback_{ij}(n) \, q_j^-(n-1) &   i \in C\\
q_i^-(n) &=& 1 & i \in A \\
q_i^-(n) &=& 0 & i \in B 
\end{array}\right.
\end{equation}
with initial condition $q_i^-(0) = \1_A(i)$.
\end{Thm}
The proof uses some of the arguments of the proofs above for the stationary, infinite-time and periodic infinite-time cases and can be found in the Appendix \ref{app:proofs}.

The following theorem provides us with an analogue result to Theorem \ref{Lm:q_fb_all_paths} for the forward and the backward committors of a finite-time Markov chain written in terms of path probabilities of the paths that start in $A$ and end in $B$ within the restricted time frame $\T$.
\begin{Lemma} \label{Lm:q_fb_finite_all_paths}
The forward committor at time $n \in \{0, \dots, N-2\}$ and the backward committor at time ${n \in \{1, \dots, N-1\}}$, respectively, satisfy for $i\in C$ the following equalities:
\begin{align} \label{eq:q_f_finite_all_paths}
q_i^+(n) 
&= \sum\limits_{\tau=n+1}^{N-1} \sum_{\substack{
i_\tau \in B \\
i_{n +1}, \dots, i_{\tau-1} \in C
}} \trans_{i i_{n+1}}(n) \cdots \trans_{i_{\tau -1} i_\tau}(\tau -1) \\
q_i^-(n)
&= \sum\limits_{\tau=0}^{n-1} \sum_{\substack{
i_{\tau} \in A \\
i_{\tau + 1}, \dots, i_{n-1} \in C
}} \transback_{i i_{n-1}}(n) \cdots \transback_{i_{\tau +1} i_{\tau}}(\tau +1)\label{eq:q_b_finite_all_paths} .
\end{align}
\end{Lemma}
The proof follows from rewriting $q_i^+(n)$ for any time $n \in \{0, \dots, N-2\}$ into a decomposition of the probabilities of all possible paths that reach $B$ within the time interval $\{n+1,\dots,N-1\}$ and rewriting $q_i^-(n)$ for any time ${n \in \{1, \dots, N-1\}}$ into a decomposition of the probabilities of all possible paths that came from $A$ within $\{1, \dots, n-1\}$. 
\remark{Similar as in the periodic case, it is possible to extend the approach to  time-dependent sets (i.e., space-time sets) $A(n)$ and $B(n)$. For instance, in order to study transitions that leave a set at a certain time (e.g., at time $n=0)$ and arrive at a certain time (e.g., at $n=N$) we can choose $A(n) = \1_{\{0\}}(n) A$  and~$B(n) = \1_{\{N\}}(n) B$. }\label{rem:space_time_set_finite}

\subsection{Transition Statistics and their Interpretation} \label{sec:stat_finite}
We have seen that the forward and backward committors for a finite-time Markov chain can be computed from the iterative equations \eqref{eq:q_f_finite} and \eqref{eq:q_b_finite} with final respectively initial conditions.  Based on these, we will next introduce the corresponding transition statistics.

The distribution of reactive trajectories (Definition \ref{def_reacdist}) is defined for any time ${n \in \{0, \dots, N-1 \}}$, and by Theorem \ref{thm_reacdist} it is given by
\[
\reacdist_i(n) = q_i^-(n)  \ \dist{i}{n} \ q_i^+ (n) .
\]
Observe that $\reacdist_i(0) = \reacdist_i(N-1) = 0$ because there are no reactive trajectories at these times. 
Thus the distribution of reactive trajectories cannot be normalized at times $0$ and $N-1$.
 As a consequence, the normalized distribution of reactive trajectories  $ \reacdistnorm(n)$ is just defined for times $n \in \{1, \dots, N-2\}$.
 
The current of reactive trajectories is defined for any time $n \in \{0, \dots, N-2\}$, and it is given by (Theorem \ref{thm_reaccurr})
\[ \current_{ij}(n) = q_i^-(n) \  \dist{i}{n}  \ \trans_{ij}(n) \  q^+_j(n+1) .\]
Similarly, the effective current of reactive trajectories $\effcurrent_{ij}(n) :=\max\{ \current_{ij}(n) - \current_{ji}(n),0\}$ is defined only for times $n \in \{0, \dots, N-2\}$.

Also in this case, the current satisfies certain conservation principles: 
\begin{Thm} \label{Thm:conservationlaw_finite}
For a finite-time Markov chain $(X_n)_{0 \leq n \leq N-1}$ satisfying Assumption \ref{ass:finite_chain}, the reactive current flowing   out of the node $i \in C$ at time $n$ equals the current flowing  into a node $i \in C$ at time $n-1$, i.e.,
\begin{equation}
\sum_{j\in\St}  \current_{ij}(n) = \sum_{j\in\St} \current_{ji}(n-1) 
\end{equation}
 for $n \in \{1, \dots, N-2\}$.
Further, the reactive current flowing out of $A$ into $\St$ over the whole time period $\{0,\dots,N-2\}$ equals the flow of reactive trajectories from $\St$ into $B$ over the period
\begin{equation}\label{eq:periodic_current_outin}
\sum\limits_{n=0}^{N-2} \sum_{i\in A, j\in\St} \current_{ij}(n) = \sum\limits_{n=0}^{N-2} \sum_{ i\in\St, j\in B} \current_{ij}(n).
\end{equation} 
\end{Thm}
The proof   can be found in  the Appendix \ref{app:proofs}.

The discrete rate of transitions leaving $A$ is defined for times $n \in \{0, \dots, N-2\}$
\[
\rateA(n) = \sum_{i\in A, j\in\St} \current_{ij}(n)
\]
 whereas the  discrete rate of transitions entering $B$ is defined for times ${n \in \{1, \dots, N-1\}}$
\[
\rateB(n) = \sum_{i\in \St, j\in B} \current_{ij}(n-1) .
\]
By plugging the definitions of the rates into result \eqref{eq:periodic_current_outin} of Theorem \ref{Thm:conservationlaw_finite} and by re-indexing the times, we note 
that the discrete departure rate averaged over the times ${n \in \{0, \dots, N-2\}}$ equals the time-averaged discrete arrival rate over the times ${n \in \{1, \dots, N-1\}}$, which we denote by $\bar{k}^{AB}_{N}$, i.e.,
\[ \bar{k}^{AB}_{N} := \frac{1}{N}\sum_{n=0}^{N-2} \rateA(n) = \frac{1}{N}\sum_{n=1}^{N-1} \rateB(n) . \]

In the infinite-time, stationary case, Theorem \ref{thm_traj_average}  tells us that $k^{AB}$ equals the time average of the number of reactive pieces departing per time step along a single infinitely long trajectory.  Here, we cannot apply the ergodic theorem to turn $\bar{k}^{AB}_{N}$ into an average along a single trajectory. Instead,
we can write the time-averaged rate in terms of an ensemble of trajectories to get a better understanding.
For this, we take $K\in\mathbb{N}$ i.i.d.\ realizations of the finite-time chain, i.e., each sample $\smash{ (\hat{X}_n^i)_{n\in \{0,\dots,N-1\}} }$ is distributed according to the law of the finite-time dynamics with $\hat{X}_0^i \sim \lambda(0)$. Then we have by the law of large numbers:

\begin{equation} 
\begin{split}\label{equ:avrate_MC}
\bar{k}^{AB}_{N}
&=\frac{1}{N} \sum\limits_{n=0}^{N-2}  \prob(X_n\in A, \tau_B^+(n+1) < \tau_A^+(n+1)) \\
&=\frac{1}{N} \sum\limits_{n=0}^{N-2}  \E \left[ \1_A(X_n) \1_B\left(X_{\tau^+_{A\cup B}(n+1)}\right)\right] \\
&= \lim_{K\to \infty}  \frac{1}{N} \sum\limits_{n=0}^{N-2}\frac{1}{K} \sum\limits_{i=1}^{K}  \1_A(\hat{X}_n^i) \1_B\left(\hat{X}^i_{\tau^+_{A\cup B}(n+1)}\right).
\end{split}
\end{equation}
Further, we can rewrite the second line of \eqref{equ:avrate_MC} as 
\begin{equation}  
\begin{split}
\bar{k}^{AB}_{N}
&=\frac{1}{N} \sum\limits_{n=0}^{N-2}  \E \left[ \1_A(X_n) \1_B\left(X_{\tau^+_{A\cup B}(n+1)}\right)\right] = \frac{1}{N}  \E \left[\sum\limits_{n=0}^{N-2}  \1_A(X_n) \1_B\left(X_{\tau^+_{A\cup B}(n+1)}\right)\right]
\end{split}
\end{equation}
and thus giving the average rate $\bar{k}^{AB}_{N}$ the interpretation of the total expected  amount of reactive trajectories within $\{0,\dots,N-1\}$ divided by the number of time steps. Analogously, we can apply the same argument for the time-averaged probability of being on a transition
\[
\bar{Z}_N^{AB} := \frac{1}{N} \sum_{n=0}^{N-1} Z^{AB}(n) = \frac{1}{N}\E \left[\sum\limits_{n=0}^{N-1} \1_A\left(X_{\tau^-_{A\cup B}(n)}\right) \1_B\left(X_{\tau^+_{A\cup B}(n)}\right)\right],
\]
which can be understood as the expected number of time steps the Markov chain is on a transition during $\{0, \dots, N-1\}$ divided by $N$.
Last, we define the ratio
\[
\avmeantime := \frac{\bar{Z}_N^{AB}}{\bar{k}_N^{AB}} 
\]
and observe that it provides us with the average expected duration of a reactive trajectory over ${n \in \{0, \dots, N-1\}}$.

\begin{remark} \label{rem:non-ergodic_processes}
In~\cite{von2018statistical}, transitions from $A$ to $B$ are studied in non-ergodic and non-stationary processes. But the focus there is put on the first pieces of the trajectories starting in $B^c$ which arrive in $B$ after having been to $A$ (called the first passage paths), whereas we consider the ensemble of all the transitions from $A$ to $B$ within the time interval $\T$. Further, the first passage paths are divided into nonreactive and reactive segments  and statistics for both,  the reactive and  the non-reactive ensemble,  are computed.  Last, their approach does not have any restrictions on the length of the first passage path and therefore is not straightforwardly extendable to finite-time processes.
\end{remark}
\subsection{Convergence Results of TPT Objects from Finite-time Systems to Infinite-time}\label{sec:convergence_finite}
The aim of this section is to show the consistency of the committors and  TPT objects between the finite-time case and infinite-time case. We will show that when assuming a time-homogeneous chain on a finite-time interval that is stationary, the TPT objects converge to the infinite-time, stationary case when letting the time interval   go to infinity.

To this end, let us consider the Markov chain $(X_n)_{-N \leq n \leq N}$  on the time interval $\{-N, \dots, N\}$ with a time-homogeneous and  irreducible transition matrix $P=(P_{ij})_{i,j \in \St}$. When choosing  the unique, strictly positive invariant density ${\pi = (\pi_i)_{i \in \St}}$ of $P$ as an initial density, the density for all ${n \in \{-N, \dots, N\}}$ is given by 
\begin{equation}\label{eq:stationary_dist}
\prob(X_n = i) = \pi_i .
\end{equation}
The time-reversed process $(\back{X}_n)_{-N \leq n \leq N}$ is also time-homogeneous and stationary, since its transition probabilities are given by
\begin{equation} \label{eq:transback_conv}
\transback_{ij} = \frac{\pi_j}{\pi_i} \, \trans_{ji}.
\end{equation}
The forward committor and the backward committor for a finite-time Markov chain on the time interval $ \{-N, \dots, N\}$ satisfy \eqref{eq:q_f_finite} and \eqref{eq:q_b_finite} with a slight adjustment of the time interval.

The next theorem provides us the desired result:
\begin{Thm} \label{thm:convergence}
The committors and transition statistics defined for an irreducible Markov chain in stationarity on a finite time interval correspond in the limit that the interval $\{-N, \dots, N\}$ tends to $\Z$ to the objects defined for a stationary, infinite-time Markov chain. For any $i, j \in \St$ it holds that
\begin{align*}
\lim_{\substack{n \in \{-N, \dots, N-1\} \\ N \rightarrow \infty}} q_i^+(n) &= q_i^+ , &
\lim_{\substack{n \in \{-N+1, \dots, N\} \\ N \rightarrow \infty}} q_i^-(n) &= q_i^- , &
\lim_{\substack{n \in \{-N+1, \dots, N-1\} \\ N \rightarrow \infty}} \reacdist_i(n) &= \reacdist_i , \\
\lim_{\substack{n \in \{-N+1, \dots, N-2\} \\ N \rightarrow \infty}} \current_{ij}(n) &= \current_{ij} , &
\lim_{\substack{n \in \{-N+1, \dots, N-2\} \\ N \rightarrow \infty}} \rateA(n) &= \rateA , &
\lim_{\substack{n \in \{-N+2, \dots, N-1\} \\ N \rightarrow \infty}} \rateB(n) &= \rateB .
\end{align*}
\end{Thm}
\begin{proof}
First, we see that the forward committor at time ${n \in \{-N, \dots, N-1\}}$ and the backward committor at time ${n \in \{-N+1, \dots, N\}}$ for finite-time Markov chains correspond in the limit that $N \rightarrow \infty$ to the forward and the backward committors defined for stationary, infinite-time Markov chains. For any ${n \in \{-N, \dots, N-1\}}$ we can see that
\begin{equation} \label{eq:q_f_lim} 
\begin{split}
\lim_{\substack{n \in \{-N, \dots, N-1\} \\ N \rightarrow \infty}} q_i^+(n)
\overset{(1)}{=}  \sum\limits_{\tau \in \Z^+} \sum_{\substack{
i_1, \dots, i_{\tau -1} \in C \\ 
i_\tau \in B
}} \trans_{i i_1} \dots \trans_{i_{\tau -1} i_\tau} 
& \overset{(2)}{=}  q_i^+ , 
\end{split}
\end{equation}
where (1) follows from applying first Lemma \ref{Lm:q_fb_finite_all_paths} and then re-indexing the coefficients of the transition matrix $P$, and (2) follows directly from Lemma \ref{Lm:q_fb_all_paths}. Analogously we obtain for any $n \in \{-N +1, \dots, N\}$ that
\begin{equation} \label{eq:q_b_lim} 
\begin{split}
\lim\limits_{\substack{n \in \{-N+1, \dots, N\} \\ N \rightarrow \infty}} q_i^-(n)
=  \sum\limits_{\tau \in \Z^-} \sum_{\substack{
i_\tau \in A \\
i_{\tau +1}, \dots, i_{-1} \in C
}} \transback_{i_{\tau +1} i_\tau} \dots \transback_{i i_{-1}} 
= q_i^- .
\end{split}
\end{equation}
Hence, by putting together \eqref{eq:q_f_lim}, \eqref{eq:q_b_lim} and \eqref{eq:stationary_dist} we show that for any $i, j \in \St$
\begin{align}
\lim_{\substack{n \in \{-N+1, \dots, N-1\} \\ N \rightarrow \infty}} \reacdist_i(n) &= q_i^-(n) \, \prob(X_n = i)\, q_i^+(n) = q_i^- \, \pi_i \, q_i^+ = \reacdist_i , \\
\lim_{\substack{n \in \{-N+1, \dots, N-2\} \\ N \rightarrow \infty}} \current_{ij}(n) &= q_i^-(n) \, \prob(X_n = i) \, \trans_{ij} \, q_j^+(n+1) = q_i^- \, \pi_i \, \trans_{ij} \, q_j^+ = \current_{ij} , \\
\lim_{\substack{n \in \{-N+1, \dots, N-2\} \\ N \rightarrow \infty}} \rateA(n) &=  \sum_{\substack{i \in A \\ j \in \St}} \current_{ij}(n) = \sum_{\substack{i \in A \\ j \in \St}} \current_{ij} = \rateA , \\
\lim_{\substack{n \in \{-N+2, \dots, N-1\} \\ N \rightarrow \infty}} \rateB(n) &= \sum_{\substack{i \in \St \\ j \in B}} \current_{ij}(n-1) = \sum_{\substack{i \in \St \\ j \in B}} \current_{ij} = \rateB .
\end{align}
\end{proof}

\subsection{Numerical Example 3: Time-homogeneous, Finite-time Dynamics  on a 5-state Network}
\label{sub:network_finite_stationary}
\begin{figure}
\begin{subfigure}[t]{0.425\textwidth}
\centering
\includegraphics[width=1\textwidth]{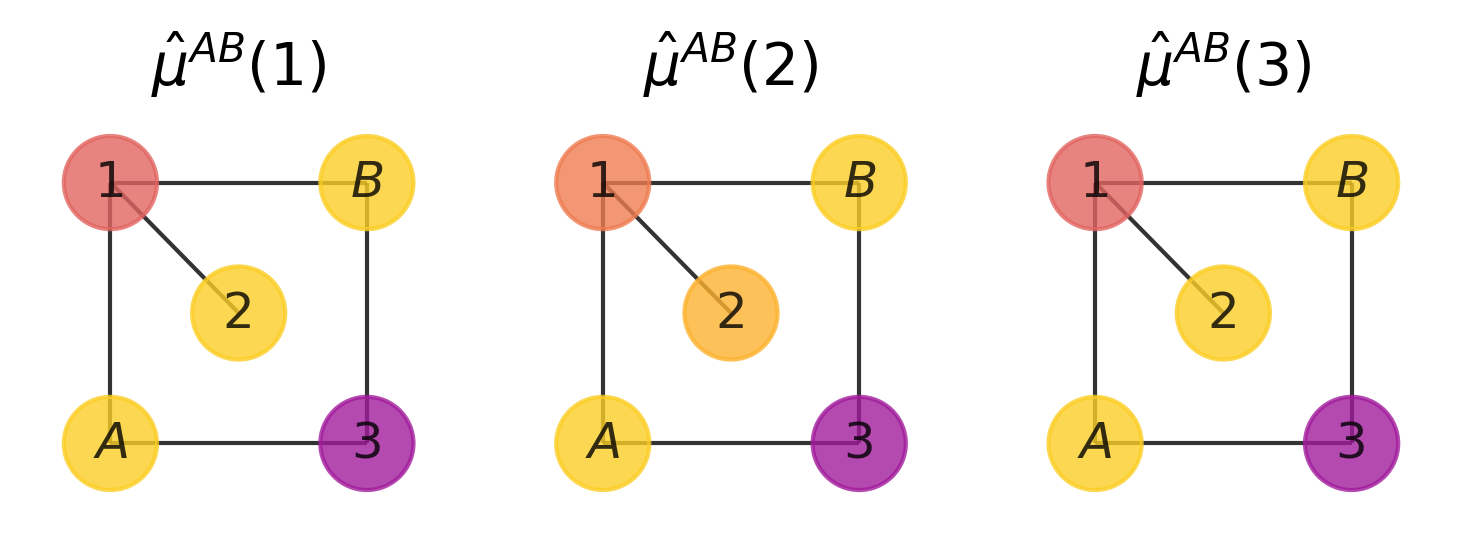}
\caption{}
\end{subfigure}
\begin{subfigure}[t]{0.53\textwidth}
\centering
\includegraphics[width=1\textwidth]{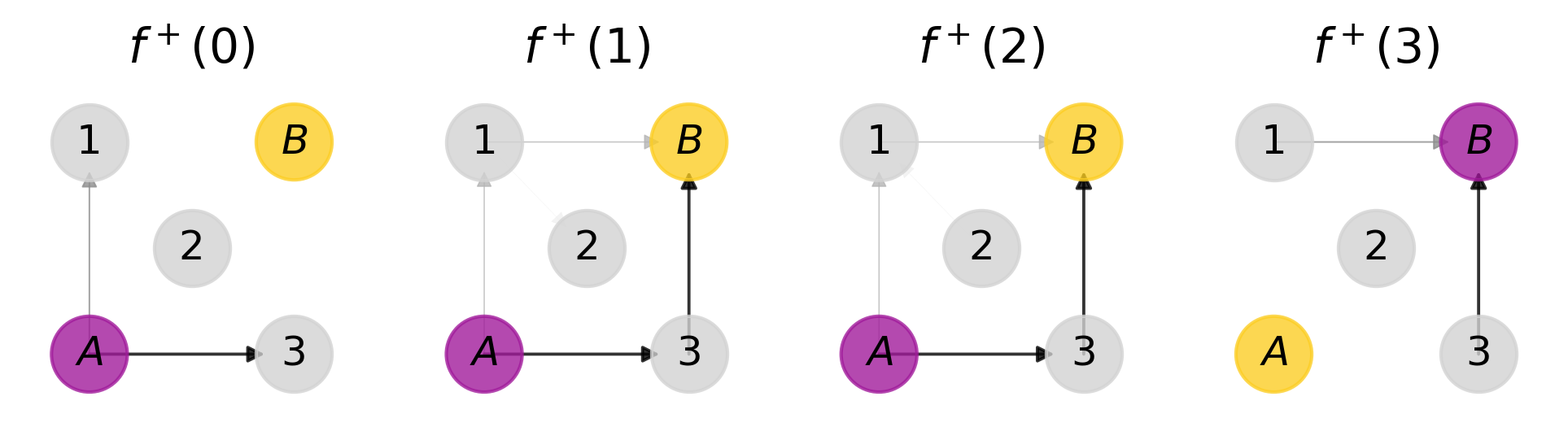}
\caption{}
\end{subfigure}
\begin{subfigure}[t]{0.03\textwidth}
\centering
\includegraphics[width=1\textwidth]{img/small_network_colorbar.png}
\end{subfigure}
\caption{The transition statistics of the finite-time system of \hyperref[sub:network_finite_stationary]{Example 3} are shown for times $n\in \{0,\dots, 4\}$. In  \textbf{(a)} the node color indicates the value of the normalized distribution, ranging from yellow (low values) via orange to purple (high value). In  \textbf{(b)} the strength of the  effective current from one state $i$ to state $j$ is shown by the thickness and greyscale of the arrow pointing from $i$ to $j$, and rate $\rateA_n$  respectively $\rateB_{n+1}$ are shown by the color of the nodes $A$ and $B$ of the effective current plot.
} \label{fig:example_f}
\end{figure}
We consider a time-homogeneous Markov chain over $\T=\{0,1,\dots,N-1\}$, $N=5$ with the transition rules given by $\trans(n) = P$ for all $n \in \{0,1,\dots,N-2\}$. $P$ is the transition matrix used in \hyperref[sub:network_ergodic]{Example 1} and the initial density is given by the stationary density $\pi$ of $P$. Thus we are in the same setting as before in  \hyperref[sub:network_ergodic]{Example 1}, but   we are now looking at transitions between $A$ and $B$ that are restricted to take place within the finite time window.

The transition statistics computed for this example are shown in Figure \ref{fig:example_f}.
Since the time interval was chosen rather small,  transitions via $1$, that usually do a detour to the metastable state $2$, are very unlikely. The transition path via $3$ is more likely, we can see that the effective current through $3$ is much stronger, also the normalized reactive distribution indicates that it is the most likely to be in node $3$ when reactive. We also note that the rate $\bar{k}_5^{AB}=0.005$ is much smaller than in the stationary case, only few transitions from $A$ to $B$ are completed (on average) in the short time frame $\{0,1,..,N-1\}$.

\subsection{Numerical Example 4: Time-inhomogeneous, finite-time dynamics  on a 5-state network}
\label{sub:network_finite_inhom}
\begin{figure}
\centering
\begin{subfigure}[c]{0.3\textwidth}
\centering
\includegraphics[width=1\textwidth]{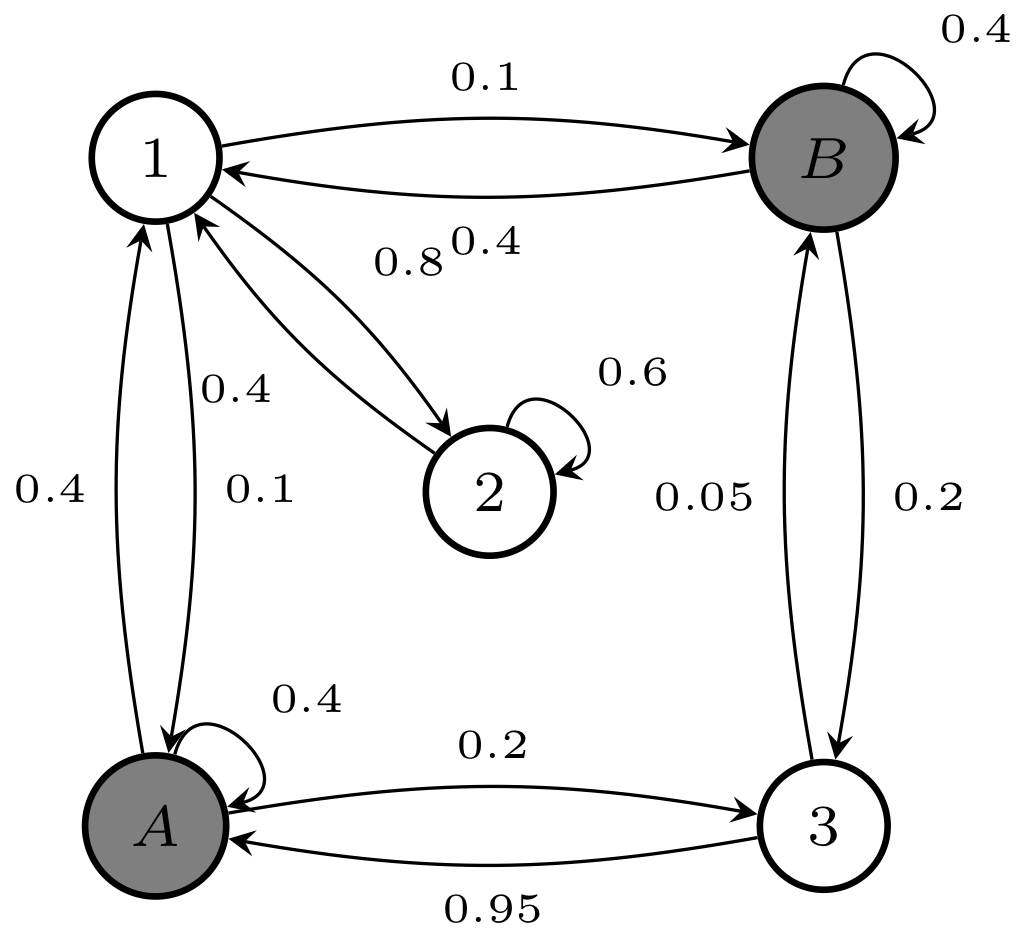}
\caption{}
\end{subfigure}
\begin{subfigure}[c]{0.265\textwidth}
\centering
\includegraphics[width=1\textwidth]{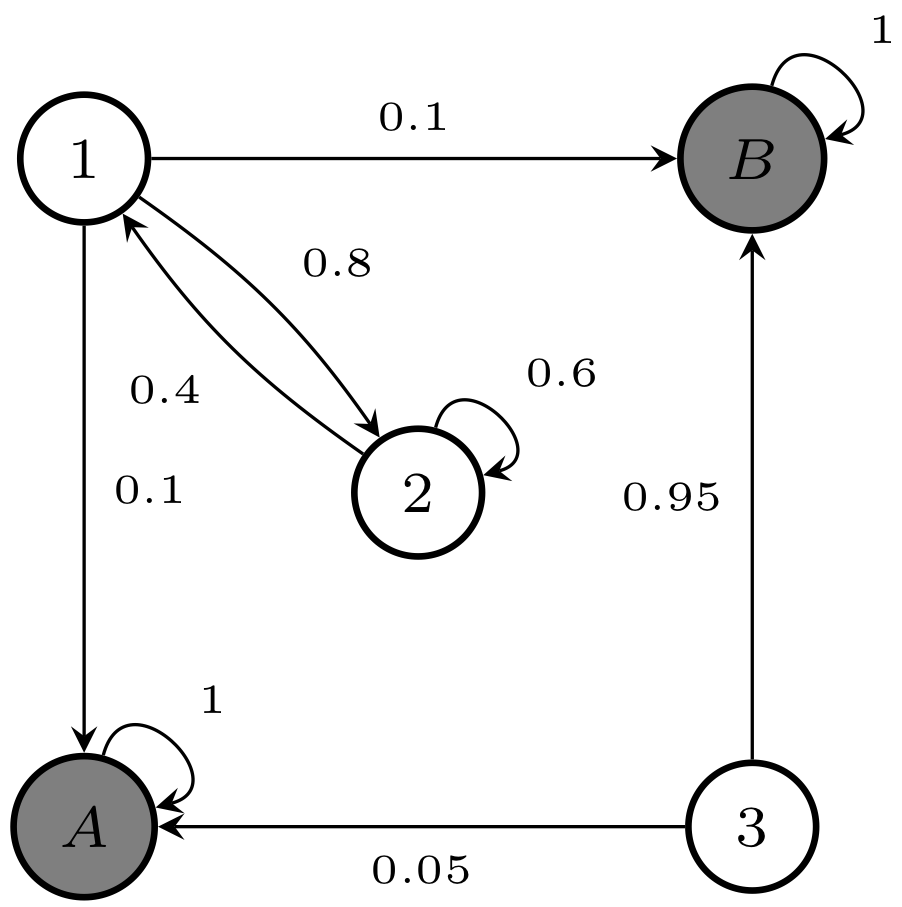}
\caption{}
\end{subfigure}
\caption{The dynamics of the finite-time, time-inhomogeneous system  (\hyperref[sub:network_finite_inhom]{Example 4})  are given by \textbf{(a)} $P(n) = P + K$ at even times, and \textbf{(b)} $P(n) = P - K$ at odd times.  }
\label{fig:example_inhom_trans}
\end{figure}

\begin{figure}
\centering
\begin{subfigure}[t]{0.425\textwidth}
\centering
\includegraphics[width=1\textwidth]{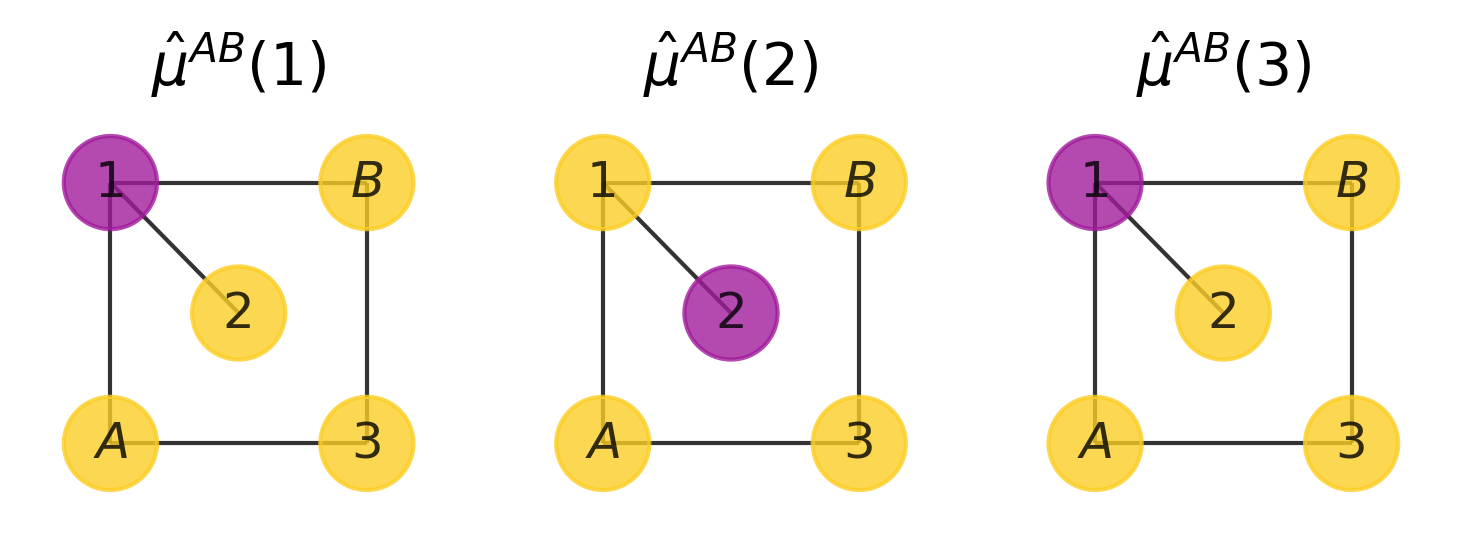}
\caption{}
\end{subfigure}
\begin{subfigure}[t]{0.53\textwidth}
\centering
\includegraphics[width=1\textwidth]{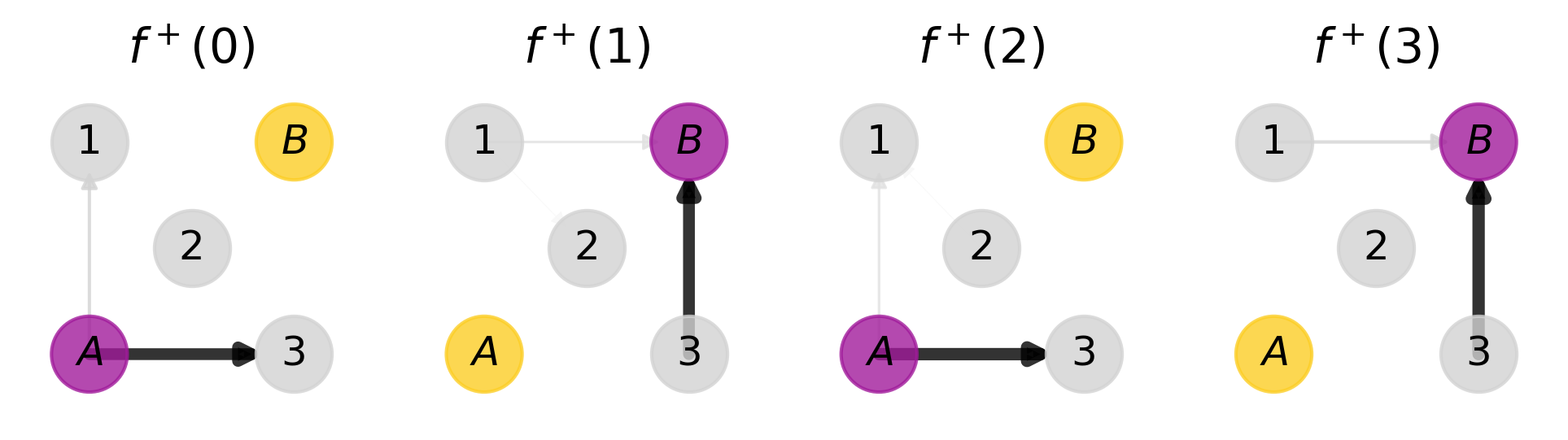}
\caption{}
\end{subfigure}
\begin{subfigure}[t]{0.03\textwidth}
\centering
\includegraphics[width=1\textwidth]{img/small_network_colorbar.png}
\end{subfigure}
\caption{  Here we show \textbf{(a)} the normalized distribution and  \textbf{(b)} the current of transitions from $A$ to $B$ for the finite-time, time-inhomogeneous system  (\hyperref[sub:network_finite_inhom]{Example 4}).}
\label{fig:example_inhom}
\end{figure}
As a next example, we consider  a time-inhomogeneous chain over the finite time window $\T= \{0, \dots, N-1\}$, $N=5$. We again set the stationary density of $P$ as an initial density and let the transition matrices depend on the time but in such a way that $\frac{1}{N-1}\sum_n P(n) = P$. For any $n \in \{0, \dots, N-2\}$ let
\[
P(n) = 
\left\{\begin{array}{ll} 
P + K & \text{if} \quad n \; (\text{mod} \; 2) = 0 \\
P - K & \text{if} \quad n \; (\text{mod} \; 2) = 1 
\end{array}\right.
\]
where $K$ is the 0-rowsum matrix as given in  Figure \ref{fig:example_inhom_trans}.
At times $n = 0, 2$ the transition matrices become $P + K$ and the subsets $A$ and $B$ are less metastable. At times $n=1, 3$ the transition matrices become $P - K$ and $A$ and $B$ are more metastable.

The computed transition statistics are shown in Figure \ref{fig:example_inhom}. The effective current plot shows that the majority of reactive trajectories are leaving $A$ at times $n=0, 2$ and go to $3$ and at times $n=1, 3$ they move from $3$ to $B$. The time-averaged rate for the finite-time, time-homogeneous case from \hyperref[sub:network_finite_stationary]{Example 3} is $\bar{k}_{5}^{AB} = 0.005$ while the time-averaged rate for this time-inhomogeneous case is~$\bar{k}_{5}^{AB} =0.012$, more than two times larger. We thus demonstrated that by adding a forcing to the time-homogeneous dynamics that changes the metastability of $A$ and $B$, we can increase the rate of transitions,  even though the forcing vanishes on average. 
We note that utilizing perturbations that tip a system out of equilibrium is used in statistical mechanics to accelerate convergence of statistics~\cite{ hamelberg2004accelerated,LePa13}.

\subsection{Numerical Example 5: Growing Finite-time Window for the 5-state Network}
\label{sub:network_conv}
Let us consider the stationary Markov chain introduced in \hyperref[sub:network_finite_stationary]{Example 3} but on the time interval $\T = \{-N, \dots, N\}$. We want to see numerically that the forward respectively backward committor at time $n$ converges under the $l^2$-norm to the forward respectively backward committor of the infinite-time, stationary system by extending the time interval ${\T = \{-N, \dots, N\}}$ for $N$ big enough such that $N \pm n \gg 1$, i.e., 
\begin{align*}
\lim_{\substack{n \in \{-N, \dots, N-1\} \\ N \pm n\gg 1}} ||q^+ - q^+(n)||_2  \approx 0, \ &
\lim_{\substack{n \in \{-N +1, \dots, N\} \\ N \pm n \gg 1}} ||q^- - q^-(n)||_2  \approx 0.
\end{align*}
In Figure \ref{fig:conv_finite} we show this convergence numerically, note that the statistics will not necessary converge near the boundary of the interval.
\begin{SCfigure}[1][!h]
\centering
\includegraphics[width=0.5\textwidth]{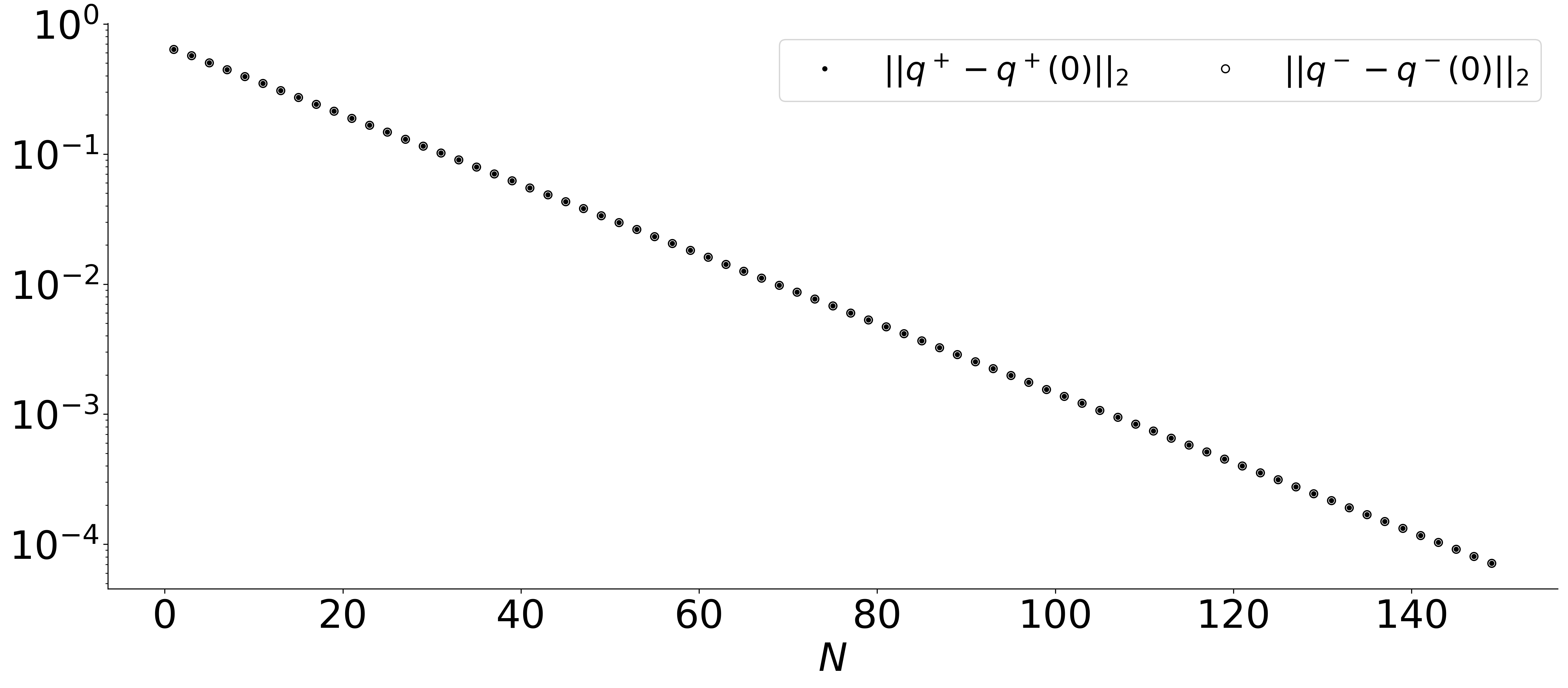}
\caption{The committors at time ${n=0}$ of the finite-time chain with stationary dynamics (\hyperref[sub:network_conv]{Example 5}) converge under the $l^2$-norm to the committors of the infinite time chain (\hyperref[sub:network_ergodic]{Example 1}). The $l_2$-errors decrease at an exponential rate.}

\label{fig:conv_finite}
\end{SCfigure}

\section{Numerical Examples in a Triplewell Landscape}
\label{sec:triplewell_ex}
In this section, we exemplarily study the transition behaviour of a particle diffusing in a triple well energy landscape for several scenarios: for infinite-time, stationary dynamics (Section~\ref{sub:triplewell_ex_inf}), for infinite-time dynamics with an added periodic forcing (Section~\ref{sub:triplewell_ex_periodic}), and for stationary, finite-time dynamics (Sections~\ref{sub:triplewell_ex_finite_hom}, \ref{sub:well_ex_bif}).
 
\subsection{Numerical Example 6: Infinite-time, stationary dynamics}
\label{sub:triplewell_ex_inf}
\begin{figure}
\centering
\begin{subfigure}[c]{0.29\textwidth}
\includegraphics[width=1\textwidth]{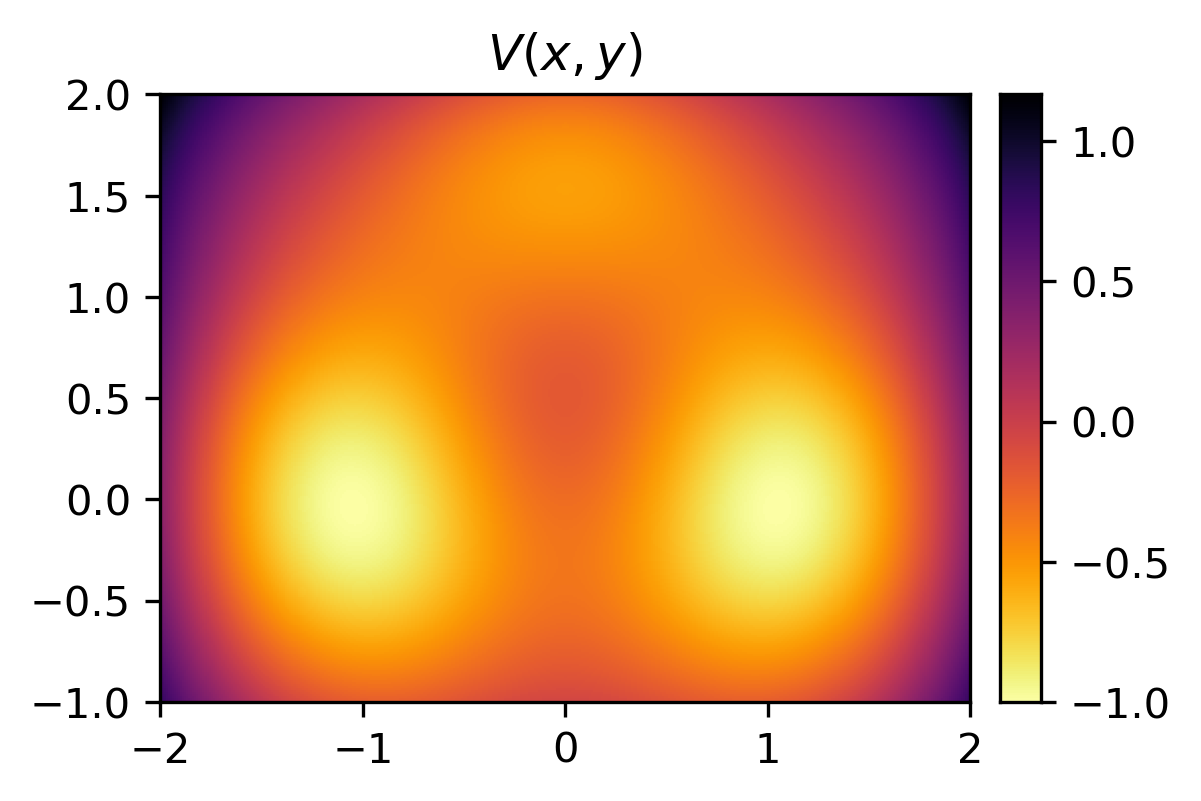}
\caption{}
\end{subfigure}
\begin{subfigure}[c]{0.7\textwidth}
\centering
\includegraphics[width=1\textwidth]{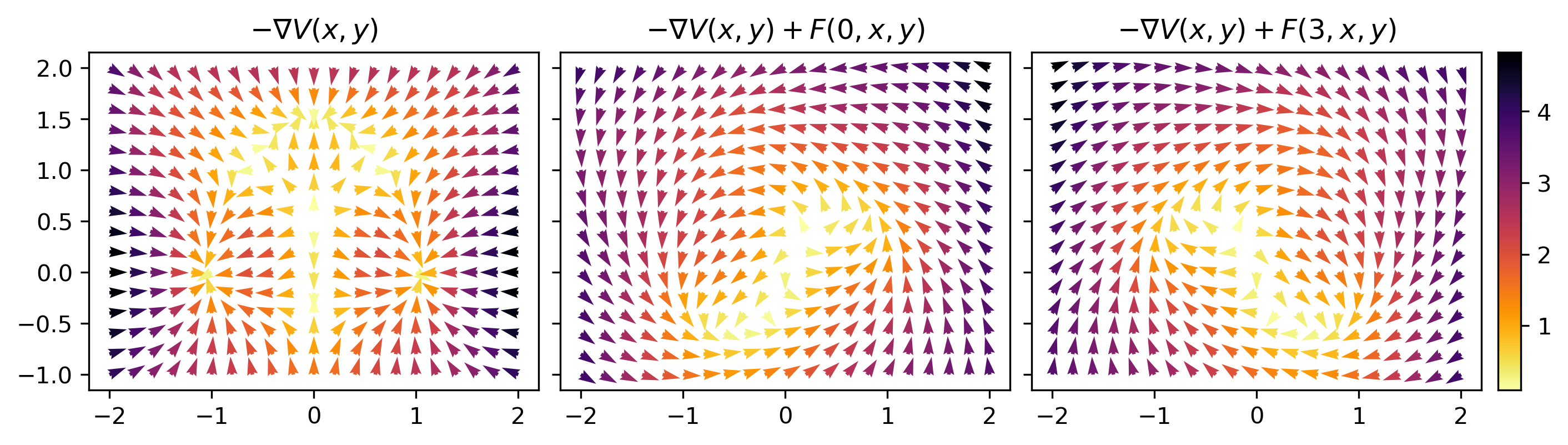}
\caption{}
\end{subfigure}
\caption{\textbf{(a)} Triplewell landscape  $V(x,y)$ and the  resulting force   when unperturbed and   when periodically perturbed. \textbf{(b)} The negative gradient $-\nabla V(x,y)$ is the force that each particle in \hyperref[sub:triplewell_ex_inf]{Example 6} without pertubation feels, while in \hyperref[sub:triplewell_ex_periodic]{Example 7} a periodically varying circulative force is added resulting at times   $m=0$ and   $m=3$ in the shown forces.}
\label{fig:potential}
\end{figure}

\begin{figure}
\centering
\begin{subfigure}[c]{0.245\textwidth}
\includegraphics[width=1\textwidth]{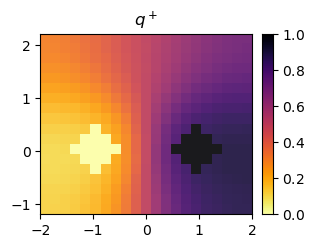}
\end{subfigure}
\begin{subfigure}[c]{0.245\textwidth}
\includegraphics[width=1\textwidth]{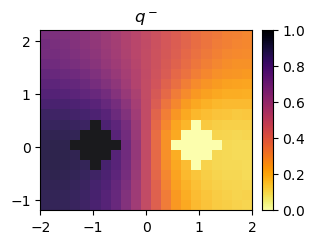}
\end{subfigure}
\begin{subfigure}[c]{0.245\textwidth}
\includegraphics[width=1\textwidth]{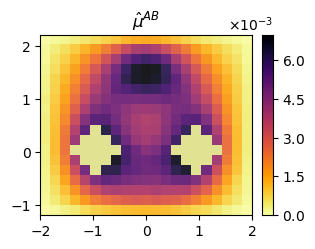}
\end{subfigure}
\begin{subfigure}[c]{0.235\textwidth}
\includegraphics[width=1\textwidth]{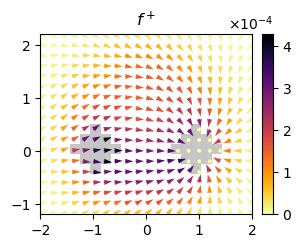}
\end{subfigure}
\caption{TPT objects for stationary, infinite-time dynamics in the triple well landscape as explained in \hyperref[sub:triplewell_ex_inf]{Example 6}. The last plot shows the accumulated $\effcurrent$ for each state $i$ (see the footnote).} \label{fig:well_example}
\end{figure}

We consider the diffusive motion of a particle $(X_t)_{t\in\R}$ in $\R^2$ according to the overdamped Langevin equation
    \begin{equation}\label{eq:langevin}
        dX_t = - \nabla V(X_t) dt + \sigma dW_t
    \end{equation} 
    where $\sigma>0$ is the diffusion constant, $W_t$ is a $2-$dimensional standard Brownian motion (Wiener process) and $V(x,y): \R^2 \rightarrow \R$ is the triple well potential    given in~\cite{metzner2009transition, schutte2013metastability},
    \begin{equation}
       \begin{split} V(x,y) = & \frac{3}{4} \exp\left( -x^2-\left(y-\frac{1}{3}\right)^2\right) - \frac{3}{4} \exp\left(  -x^2-\left(  y-\frac{5}{3}\right)^2\right) \\
       &-\frac{5}{4} \exp\left(  -\left(  x-1\right)^2 - y^2\right) - \frac{5}{4} \exp\left(  -\left(  x+1\right)^2-y^2\right) + \frac{1}{20} x^4 + \frac{1}{20} \left(  y-\frac{1}{3}\right)^4
       \end{split}
    \end{equation} and shown in Figure~\ref{fig:potential}. The particle is pushed by the force $-\nabla V(x,y)$ ``downhill'' in the energy landscape, while simultaneously being subject to a random forcing due to the Brownian motion term.
    The stationary density of the process is given by the Boltzmann distribution $$\pi(x,y) = Z^{-1} \exp(-2\sigma^{-2} V(x,y))$$
    with normalization factor $Z$.
    
    Before applying  TPT on this example, we have to discretize the process into a Markov chain. 
    We want to estimate a transition  matrix that gives us the probability to jump between discrete cells of the state space $[-2,2]\times [-1,2]$, here we choose regular square grid cells $\{A_i,i=1,...\}$ of size $0.2 \times 0.2$. By means of Ulams method (\cite{ulam1960collection}, see also~\cite[Chapter 2.3]{koltai2011efficient} for a summary of the method)
    one can project the dynamics of~\eqref{eq:langevin} onto the space spanned by indicator functions on the grid cells, and then further approximate the projected dynamics by a Monte-Carlo sum of sampled trajectory snippets
      \begin{equation}
        \begin{split}
            \trans_{ij} & \approx \frac{1}{C} \sum_{c=1}^C \1_{A_j}(\hat{Y}^c),
        \end{split}
    \end{equation}
    where we sample $\hat{X}^c$, $c=1,\dots,C$, uniformly from the cell $A_i$ and we get $\hat{Y}^c$  by evolving the sample forward in time according to~\eqref{eq:langevin} with time step $\tau$ (e.g., by using the Euler--Maruyama discretization of the stochastic differential equation~\eqref{eq:langevin}). The resulting process defined by the transition matrix  $\trans$ is a discrete-time, discrete-space Markov chain with time steps $\tau$.\footnote{As an alternative for estimating the transition matrix for an overdamped Langevin process (such as~\eqref{eq:langevin}), one can use the square root approximation   to get a cheap estimate of the rate matrix~\cite{lie2013square,heida2018convergences} and  thereof get the transition matrix.}
    
    We are now interested in the transition behaviour between the deep wells of $V$ when the dynamics are stationary. Therefore, we choose sets $A$ and $B$ as centered at  $(-1,0)$ and $(1,0)$ respectively   and we ask: which transition path from $A$ to $B$ is more likely; via the third metastable shallow well at $(0,1.5)$, or via the direct barrier between $A$ and $B$?
    
    The computed committor functions and statistics of the reactive trajectories are shown in Figure~\ref{fig:well_example}, we chose $\sigma = 1$, $\tau = 0.3$ and sampled $C=10,000$ for estimating the transition matrix. Since the dynamics are reversible, the backward committor is just $1-q^+$. Also, we can see that the committors are close to constant inside the metastable sets (the wells) due to the fast mixing inside the wells, but vary across the barriers.
    The computed effective current\footnote{Regarding plotting $\effcurrent$: If the underlying process is a diffusion process in $\R^d$, we can estimate for each $i$ the vector of the average direction of the effective current and the amount, i.e., to each $i$ at time $n$ we can attach the vector $\sum_{j\neq i} \effcurrent_{ij} (n) v_{ij}$, where $v_{ij}$ is the unit vector pointing from the center of the grid cell $i$ to the center of the grid cell $j$ (see, e.g., Figure~\ref{fig:well_example}).} indicates that most transitions occur via the direct barrier between $A$ and $B$. The effective current through the shallow well at $(0,1.5)$ is only very small, but due to the metastability of the well, reactive trajectories taking that path are stuck there for long and  thus contributing a lot to the density $\reacdist$. The rate of transitions is $\rate = 0.0142$ and the mean transition time is $\meantime = 10.01$.

\subsection{Numerical Example 7: Periodic, Infinite-time Dynamics}
\label{sub:triplewell_ex_periodic}

\begin{sidewaysfigure}
\begin{subfigure}{1\hsize}\centering
    \includegraphics[width=0.9\hsize]{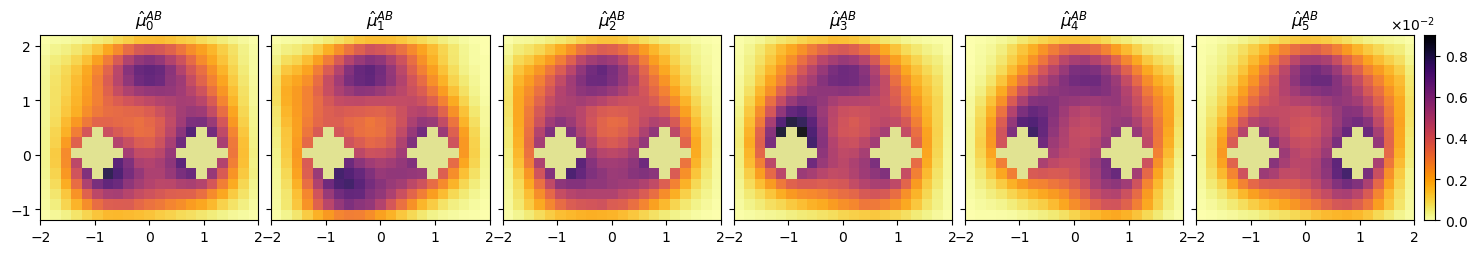}
\caption{}
\end{subfigure}%

\begin{subfigure}{1\hsize}\centering
    \includegraphics[width=0.9\hsize]{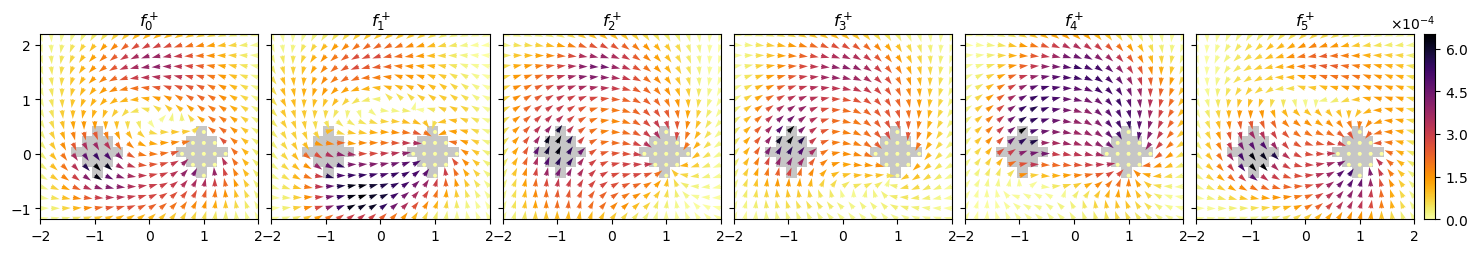}
\caption{}
\end{subfigure}

\begin{subfigure}{1\hsize}\centering
 \includegraphics[width=0.9\hsize]{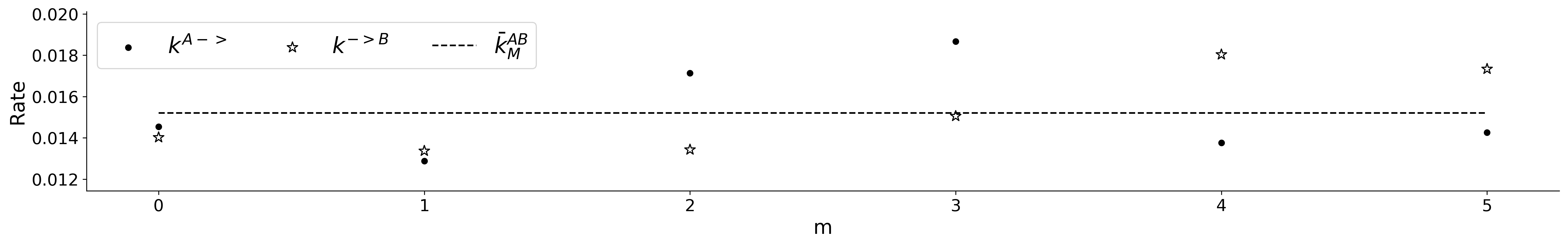}
\caption{}
\end{subfigure}

\caption{ \textbf{(a)} The density and \textbf{(b)}  effective current of reactive trajectories for the triple well dynamics with periodic circulative forcing that is sometimes clockwise, sometimes anti-clockwise, see \hyperref[sub:triplewell_ex_periodic]{Example 7}.  \textbf{(c)}  The rates $\rateA_m$ and $\rateB_m$ indicate the probability of a reactive trajectory to leave $A$ respectively enter  $B$ per time step. } 
\label{fig:well_example_p}
\end{sidewaysfigure}

Next, we are interested in studying transitions  in the triplewell landscape when a periodically varying forcing is applied. We add the force $F(x_1,x_2, t) = 1.4 \, \cos\left(\frac{2 \pi t}{ 1.8}\right) (-x_2, x_1)$, that  alternatingly due to the cosine modulation exhibits an anti-clockwise circulation   and a  clockwise circulation, to the dynamics~\eqref{eq:langevin}, resulting in the diffusion process with $1.8$-periodic forcing,
\begin{equation}
dX_t =( - \nabla V(X_t) + F(X_t, t) ) dt + \sigma dW_t.
\end{equation}   
We again discretize the dynamics and estimate transition matrices $P_0, P_1, \dots,P_{M-1}$ ($M=6$) for $\tau$-spaced time points during the period, each transition matrix is mapping $\tau=0.3$  into the future.
In Figure~\ref{fig:potential} the force from the potential plus circulation is shown for time points $m=0$ and $m=3$.

Considering the same $A$ and $B$ as before, we are now interested in the transition channels when the dynamics are equilibrated to the periodic forcing. 
The computed results are shown in Figure~\ref{fig:well_example_p}, the reactive trajectories take the lower channel via the direct barrier at the beginning of the period  due to the additional push from the forcing, and the upper channel via the shallow well towards the end of the period when the applied circulation is clockwise. It is interesting to note, that the rate of reactive trajectories leaving $A$ and entering $B$ is highest towards the end of the period, when the added forcing is clockwise and the preferred channel is the upper channel. This is contrary to the stationary example before, where reactive trajectories dominantly passed the lower channel.

\subsection{Numerical Example 8: Finite-time, Time-homogeneous Dynamics}
\label{sub:triplewell_ex_finite_hom}
\begin{figure}
\begin{subfigure}[c]{1\textwidth}
\centering
\includegraphics[width=0.8\textwidth]{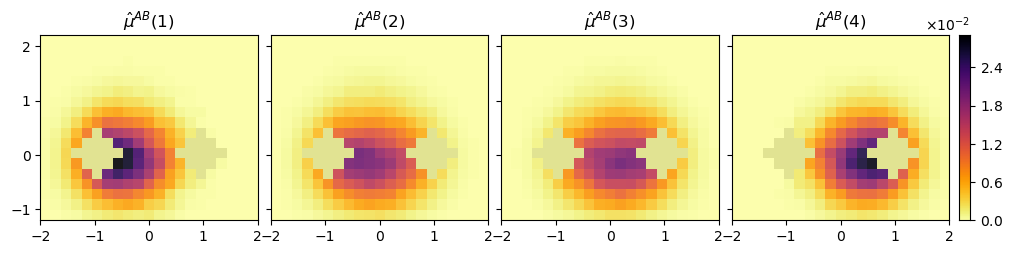}
\caption{}
\end{subfigure}
\begin{subfigure}[c]{1\textwidth}
\centering
\includegraphics[width=1\textwidth]{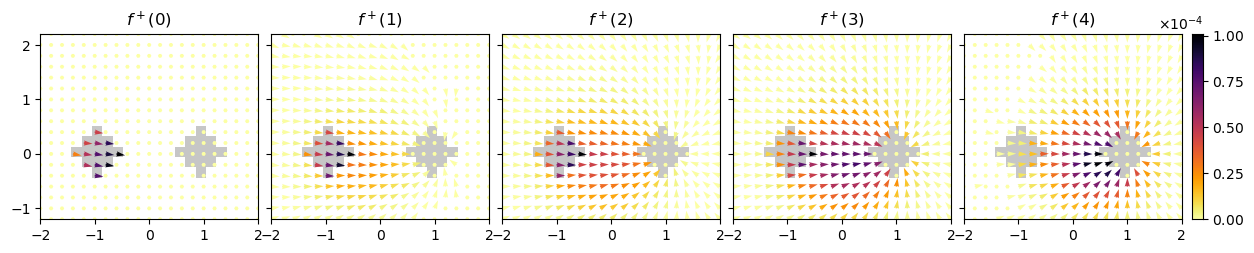}
\caption{}
\end{subfigure}
\caption{ \textbf{(a)} The density and  \textbf{(b)} effective current of reactive trajectories for the stationary triple well dynamics over a finite-time window $\{0,\dots, 5\}$, see \hyperref[sub:triplewell_ex_finite_hom]{Example 8}. Note that some quantities are not defined for all $n\in \{0,\dots,N-1\}$.} \label{fig:well_example_f}
\end{figure}

To demonstrate the effect of the finite-time restriction on the transition behaviour between $A$ and $B$, we now study the time-homogeneous triplewell dynamics restricted to the time window~$\T = \{0,\dots, N-1\}$, $N=6$,  and initiated in the stationary density.

Even though we study the same underlying dynamics as in the stationary, infinite-time case (\hyperref[sub:triplewell_ex_inf]{Example 6}), the possible transition paths between $A$ and $B$ are limited to the  pathway that is fast to traverse, i.e., the lower channel via the direct barrier, see Figure~\ref{fig:well_example_f}. Since only a small portion of the reactive trajectories from the infinite-time example \hyperref[sub:triplewell_ex_inf]{6} has a short enough transition time to be considered in this case, the average rate of transitions $\bar{k}^{AB}_6= 0.0017$ is much lower than the corresponding rate $\rate = 0.0142$ in the infinite time case (\hyperref[sub:triplewell_ex_inf]{Example 6}), and the average time a transition takes $\bar{t}^{AB}_6= 2.055$ is much shorter than in the infinite-time case.
 
\subsection{Numerical Example 9: Bifurcation Studies using Finite-time TPT}
\label{sub:well_ex_bif}
\begin{figure}
\centering
\begin{subfigure}[c]{1\textwidth}
\includegraphics[width=1\textwidth]{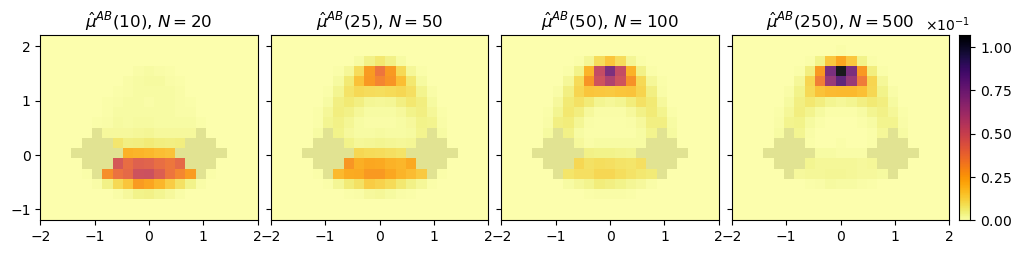}
\caption{}
\end{subfigure}

\begin{subfigure}[c]{1\textwidth}
\includegraphics[width=1\textwidth]{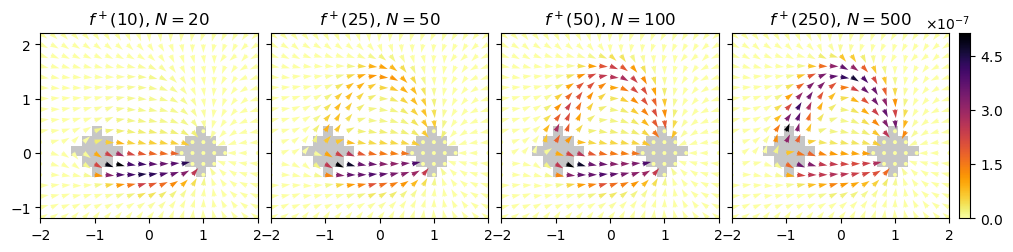}
\caption{}
\end{subfigure}

\caption{Qualitative changes in the transition behaviour for finite-time, stationary TPT and increasing time-interval lengths, see \hyperref[sub:well_ex_bif]{Example 9}. The plots show  \textbf{(a)} the density of reactive trajectories (normalized to being reactive) and  \textbf{(b)} the current at   a time point in the middle of the finite time interval.} \label{fig:well_example_bif}
\end{figure}

Last, we want to highlight the usage of finite-time TPT to study \emph{large} qualitative changes in the transition behavior of a system. We consider the stationary triple well example over a finite interval $\T=\{0,\dots,N-1\}$, but this time with a smaller noise strength of $\sigma=0.26$ compared to the previous examples.

As we increase the interval length $N$ from $N=20$ to $N=500$, we allow reactive trajectories to be longer and longer, and thus the average reactive trajectory changes.  Whereas for $N=20$ most of the density and current is around the lower transition channel, see Figure~\ref{fig:well_example_bif}, for $N=500$ most of the density  and current is around the upper transition channel. The transition behaviour restricted to the time interval of size $N=500$ is already close to   the infinite-time transition dynamics, this can be seen by comparing the case of $N=500$ with Figure~\ref{fig:potential_intro}(b), where the effective current of the same dynamics but in infinite time is depicted.

\section{Conclusion}
In this paper, we generalized Transition Path Theory such that it is applicable not only to infinite-time, stationary Markov processes but also to periodically varying and to time-dependent Markov processes on finite time intervals. We restricted our results to Markov processes on discrete state spaces and in discrete time, but generalizations should be straightforward (e.g., following~\cite{weinan2006towards,metzner2009transition}). 

The \emph{theory} is intended to generalize TPT towards applicability in, e.g., non-equilibrium molecular, climate, fluid, or social dynamics (agent-based models). In most of these applications, the problem of \emph{computing} the TPT objects arises, as the state space can be high- and also infinite dimensional. This is a non-trivial task even in the traditional TPT context stemming from molecular dynamics---where these tools have nevertheless been successfully applied. Resolving the time-dependence poses an additional computational challenge.   All this goes beyond the scope of the present work and will be addressed elsewhere. 

First results towards the application of stationary TPT in high-dimensional state spaces have already been proposed.
In~\cite{thiede2019galerkin} a workaround was given for solving the committor equations in the case of infinite-time, stationary dynamics  by a Galerkin projection which works well   as long as the dynamics of the Markov chain can be described by low-dimensional representations. Another interesting first work~\cite{khoo2019solving} goes into the direction of using  neural networks for solving  the committor equations in high dimensions.

Further, the interpretability and visualization of the transition statistics in high dimensions or for large time intervals (large $N$ or $M$) is a point of future research. In~\cite{metzner2009transition} an algorithm for computing transition channels from the effective current of reactive trajectories was proposed, a generalization for time-dependent dynamics is ongoing work.
 
An implementation of the tools developed in this paper can be found at  \url{www.github.com/LuzieH/pytpt}.

 \paragraph{Acknowledgments}
We would like to thank Alexander Sikorski and Jobst Heitzig for insightful discussions on the theory, as well as Niklas Wulkow for valuable feedback on the paper manuscript.
LH received funding by the Deutsche Forschungsgemeinschaft (DFG, German Research Foundation) under Germany's Excellence Strategy – The Berlin Mathematics
Research Center MATH+ (EXC-2046/1, project ID: 390685689). ERB and PK acknowledge support by the DFG through grant CRC 1114, projects A05 and A01, respectively.

{
\small
\bibliographystyle{myalpha}
\bibliography{references}
}
\appendix      
\section{Appendix}

\subsection{Committors for Periodic, Infinite-time Dynamics: Augmenting the State Space}\label{app:augment_periodic}
As an alternative to the approach in Section \ref{sec:com_periodic}, we can consider the dynamics and the committor equations on the augmented state space $ \St^M$. 
The augmented transition matrix $\trans_\text{Aug}$ of size $M|\St| \times M|\St|$ contains all transition matrices $\trans_0,\dots, \trans_{M-1}$  and applies them consecutively in a row 
    $$ \trans_\text{Aug} =
\begin{pmatrix}
0 &  \trans_0 &  & 0 \\
  &  \ddots & \ddots & \\
 &   &  0&\trans_{M-2}  \\
\trans_{M-1} & &   & 0
\end{pmatrix}
$$ 
 e.g., applying $\trans_\text{Aug}$ to a space-time distribution with mass only at the $0-$th time point shifts all the mass to the next time point. Since the transition matrices are row-stochastic, so is the augmented transition matrix.
 
With that set, we can take a slightly different approach for writing   the forward committor equations  as one augmented matrix equation.
Again, we need to solve the committor only on $C$, therefore we
define the augmented $|C|M \times |C|M $ matrix $\tilde{\trans}_\text{Aug}$, which encodes only the transitions inside $C$
     $$ \tilde{\trans}_\text{Aug} :=
    \begin{pmatrix}
0 & \restr{\trans_0}{C\rightarrow C}&  &  \\
 &  \ddots & \ddots& \\
 &  & 0 & \restr{\trans_{M-2}}{C\rightarrow C} \\
\restr{\trans_{M-1}}{C\rightarrow C} & &  & 0
\end{pmatrix}
$$
and the $|C|M$-dimensional vectors 
\begin{equation*}
\tilde{q}^+ :=
    \begin{pmatrix}
(q^+_{0,i})_{i\in C}\\
(q^+_{1,i})_{i\in C} \\
\vdots \\
(q^+_{M-1,i})_{i\in C}
\end{pmatrix} \text{ and }
  \tilde{b} :=
    \begin{pmatrix}
(\sum_{j\in B} \trans_{0,ij})_{i\in C}\\
(\sum_{j\in B} \trans_{1,ij})_{i\in C} \\
\vdots \\
(\sum_{j\in B} \trans_{M-1,ij})_{i\in C}
\end{pmatrix}. 
\end{equation*}
Then from \eqref{eq:q_f_periodic}, we arrive at the following equation for  $\tilde{q}^+$
\begin{equation}\label{eq:q_f_matrixeq}
\tilde{q}^+  =  \tilde{\trans}_\text{Aug} \, \tilde{q}^++ \tilde{b} \Leftrightarrow   (I - \tilde{\trans}_\text{Aug}) \, \tilde{q}^+= \tilde{b}.
\end{equation}
Note that $(I - \tilde{\trans}_\text{Aug}) $ is large but sparse  compared to using the stacked equations \eqref{eq:q_f_stacked_m}.

For the backward committor equations, we proceed similarly and arrive at the equations
\begin{equation}\label{eq:q_b_matrixeq}
\tilde{q}^-  =  \tilde{\trans}^-_\text{Aug}  \tilde{q}^- + \tilde{a} \Leftrightarrow   (I - \tilde{\trans}^-_\text{Aug}) \tilde{q}^-= \tilde{a} 
\end{equation}
where we defined 
\begin{equation*}
\tilde{\trans}^-_\text{Aug} :=
    \begin{pmatrix}
0 &&  &  \restr{\transback_0}{C\rightarrow C} \\
\restr{\transback_{1}}{C\rightarrow C} &  \ddots & & \\
 & \ddots & 0 &  \\
 & & \restr{\transback_{M-1}}{C\rightarrow C} & 0
\end{pmatrix}, \    
 \tilde{q}^- :=
    \begin{pmatrix}
(q^-_{0,i})_{i\in C}\\
(q^-_{1,i})_{i\in C} \\
\vdots \\
(q^-_{M-1,i})_{i\in C}
\end{pmatrix} \text{ and }
\tilde{a} :=
    \begin{pmatrix}
(\sum_{j\in A} \transback_{0,ij})_{i\in C}\\
(\sum_{j\in A} \transback_{1,ij})_{i\in C} \\
\vdots \\
(\sum_{j\in A} \transback_{M-1,ij})_{i\in C}
\end{pmatrix}.
\end{equation*}
 
\begin{remark}
\label{rem:switching}
This approach can also be used for studying committors and transitions in systems with stochastic switching between different regimes. Consider the dynamics of a Markov chain that can be in $M$ different regimes each described by a transition matrix $P_m$, $m\in \{1,\dots,M\}$. The probability to switch between regimes is given by $\hat{P}\in\R^{M \times M}$. The committor probabilities would give the probability to next hit $B$ and not $A$, given the chain is in a certain state and regime and assuming the dynamics are equilibrated. Note that the periodic dynamics described in this paper are just a special case of stochastic switching, where the switching is deterministic from time $m$ to time $m+1$, i.e.,  $\hat{P}_{mm'} = \delta_{mm'-1}$ modulo $M$. The regime-augmented matrix can be written as: 
$$ \trans^{\text{switch}}_\text{Aug} =
\begin{pmatrix}
\hat{P}_{11} P_1 &  \dots   & \hat{P}_{1M} P_M \\
\vdots & \ddots &  \vdots\\
\hat{P}_{M1} P_1&  \dots  & \hat{P}_{MM} P_M
\end{pmatrix}
$$ 
which is still a row stochastic matrix, since   $\sum_{m'=1}^M \hat{P}_{mm'} =1 $ for all $m=1,\dots,M$. Exactly as above, $\trans^{\text{switch}}_\text{Aug}$ can be used for finding committor probabilities.
\end{remark}

\subsection{Proofs}\label{app:proofs}
\subsubsection{Lemma \ref{Lm:q_fb_all_paths}}
\begin{proof}
Let us consider the joint probability that the process starting in $i \in \St$ at time $n \in \Z$ reaches first $B$ before $A$ at time $\tau \in \Z_n^+ \cup \{ \infty \}$:
\begin{equation*}
J_i^+(n, \tau) := \prob(\tau_B^+(n) < \tau_A^+(n), \tau_{A\cup B}^+(n) = \tau | X_n = i) .
\end{equation*}
The law of total probability lets us write the forward committor in terms of a countable sum of the above mentioned joint probability:
\begin{equation}\label{eq:J_sum_tau}
q_i^+(n) = \sum\limits_{\tau \in \Z_n^+ \cup \{\infty\}} J_i^+(n, \tau) = \sum\limits_{\tau \in \Z_n^+} J_i^+(n, \tau),
\end{equation}
where we have used that $\prob(\tau_{A \cup B}^+(n) = \infty) = 0$ due to the ergodicity of the process, which ensures that the process will arrive at some time to the subset $A \cup B$. 

Next, by applying the same arguments of the proof of Theorem \ref{Thm:q_equ} we see that the joint probability $J_i^+(n, \tau)$ satisfies the following iterative system of equations for $\tau \in \Z_{n+1}^+$:
\begin{equation} \label{eq:q_f_joint}
\left\{ \begin{array}{rcll}
J_i^+(n, \tau) &=& \sum\limits_{j \in \St} \, \trans_{ij} \, J_j^+(n+1, \tau) & i \in C \\
J_i^+(n, \tau) &=& 0 & i \in A \cup B
\end{array}\right.
\end{equation}
with initial condition $J_i^+(n, n) = \1_B(i)$. 

Last, by using \eqref{eq:q_f_joint} recursively we can compute for any $i \in C$ and $\tau \in \Z_n^+$
\begin{equation}\label{eq:J_paths}
J_i^+(n, \tau)
= \sum\limits_{i_{n+1}, \dots, i_{\tau -1} \in C} \, \trans_{ii_{n+1}} \cdots \trans_{i_{\tau -2}i_{\tau -1}} \, J_i^+(\tau -1, \tau)
= \sum_{\substack{i_{n+1} \dots i_{\tau -1} \in C \\ i_\tau \in B}} \, \trans_{ii_{n +1}} \cdots \trans_{i_{\tau -1} i_{\tau}},
\end{equation}
where we recursively used $J_i^+(k, \tau) = 0$ for any $i \in A \cup B$ and $n+1\leq k<\tau$, and in the last iteration the fact that $J_i^+(\tau, \tau) = \1_B(i)$. 

Putting the results \eqref{eq:J_sum_tau} and \eqref{eq:J_paths} together completes the proof for the forward committor. The proof for the backward committor follows by using the same arguments.
\end{proof}

\subsubsection{Theorem \ref{Thm:conservationlaw}}
\begin{proof} 
First, for any $i\in C$, we have
\begin{equation} \label{eq:calc_conservation} \begin{split}
    \sum_{j\in\St} (\current_{ij}-\current_{ji}) &=  q_i^- \statdist_i \sum_{j\in\St} \trans_{ij} q^+_j -  q^+_i   \statdist_i \sum_{j\in\St} q_j^-\transback_{ij} =  q_i^- \statdist_i  q^+_i -  q^+_i   \statdist_i   q_i^- =0 
\end{split} \end{equation}
using the definition of the time-reversed transition probabilities and the committor equations \eqref{eq:q_f_def}, \eqref{eq:q_b_def} for $i\in C$.

Second, using that $\current_{ij}=0$ if $i\in B, j\in \St$ and also if $i\in \St, j\in A$, we can compute
 \begin{equation}  \begin{split}
     \sum_{\substack{i\in \St \\ j \in \St}} \current_{ij} = \sum_{\substack{i\in A \\ j \in \St}} \current_{ij}  +\sum_{\substack{i\in C \\ j \in \St}} \current_{ij}  +\sum_{\substack{i\in B\\ j \in \St}}  \underbrace{\current_{ij}}_{ = 0}  =  \sum_{\substack{i\in \St \\ j \in A}} \underbrace{\current_{ij}}_{ = 0} +  \sum_{\substack{i\in \St \\ j \in C}} \current_{ij}  + \sum_{\substack{i\in \St\\ j \in B}} \current_{ij}   \nonumber
\end{split} \end{equation}
and by the current conservation \eqref{eq:calc_conservation} for $i\in C$,
$\sum_{\substack{i\in C \\ j \in \St}} \current_{ij} = \sum_{\substack{i\in C \\ j \in \St}} \current_{ji}=\sum_{\substack{j\in C \\ i \in \St}} \current_{ij}$,
we arrive at
\begin{equation}  \begin{split}
      \sum_{\substack{i\in A \\ j \in \St}} \current_{ij}  +\sum_{\substack{i \in \St\\j\in C}} \current_{ij}   =  \sum_{\substack{i\in \St \\ j \in C}} \current_{ij}  +\sum_{\substack{i\in \St\\ j \in B}} \current_{ij}  \nonumber 
 \end{split} \end{equation}
 implying that $\sum_{\substack{i\in A \\ j \in \St}} \current_{ij}   =   \sum_{\substack{i\in \St\\ j \in B}} \current_{ij}.$
\end{proof}

\subsubsection{Theorem \ref{Thm:q_fb_periodic}}
   \begin{proof}
    For $i\in C$, the forward committor at time $n=m$ modulo $M$ reads
     \begin{equation} \label{eq:q_f_periodic_proof}
     \begin{split}
         q^+_{i}(n) &= \prob(\tau^+_B(n)<\tau^+_A(n)| X_n=i) \overset{(1)}{=} \prob(\tau^+_B(n+1)<\tau^+_A(n+1)| X_n=i) \\
         &\overset{(2)}{=} \sum_{j\in\St} \prob(\tau^+_B(n+1)<\tau^+_A(n+1), X_{n+1}=j| X_n=i)  \\
         &\overset{(3)}{=} \sum_{j\in\St} \prob(\tau^+_B(n+1)<\tau^+_A(n+1) |X_{n+1}=j, X_n=i) \prob( X_{n+1}=j| X_n=i)  \\
         &\overset{(4)}{=} \sum_{j\in\St} \prob(\tau^+_B(n+1)<\tau^+_A(n+1) |X_{n+1}=j) \trans_{n,ji} = \sum_{j\in\St} q^+_{j}(n+1) \trans_{n,ij} 
     \end{split} \end{equation}
     using (1) that $\tau_B^+(n), \tau_A^+(n)\geq n+1$ for $i\in C$, (2) the law of total probability, (3) the definition of conditional probabilities, (4) the strong Markov property. 
     
     And for $i\in A $ at time $n$ it follows from the definition that $\tau^+_A(n)=n$, whereas $\tau^+_B(n)>n$, thus $q^+_{i}(n)=0$, in a similar way for $i\in B$, $\tau^+_A(n)>n$,  $\tau^+_B(n)=n$ thus $q^+_{ i}(n)=1$.
     
     The proof for the backward committor equations follows the same lines.
    \end{proof}
    
\subsubsection{Lemma \ref{Lm:ex_uniq_periodic}}
\begin{proof}  
We start with the  case of the forward committor \eqref{eq:q_f_periodic}.\\
 We can rewrite \eqref{eq:q_f_stacked_m} with $m=0$, with   $\restr{\trans_0}{I\rightarrow J}$ denoting the restriction of the matrix $P_0$ to entries from $i\in I$ to $j\in J$, as the matrix equation
    \begin{equation}  \begin{split}  \label{eq:committor_stacked}
         (q^+_{0,i_0})_{i_0\in C}
          &=\underbrace{\restr{\trans_0}{C\rightarrow C} \cdots \restr{\trans_{M-1}}{C\rightarrow C} }_{=: \PC} (q^+_{0,i_M})_{i_M\in C} \nonumber  + \sum_{\tau=1}^{M}   \restr{\trans_0}{C \rightarrow C}  \cdots\restr{\trans_{\tau-1}}{C \rightarrow B} (1)_{B} \nonumber  
     \end{split} \end{equation}
    equivalently 
     \begin{equation}  \begin{split}  \label{eq:committor_stacked2}
         (I-D)(q^+_{0,i_0})_{i_0\in C}
          &=  \sum_{\tau=1}^{M}   \restr{\trans_0}{C \rightarrow C}  \cdots\restr{\trans_{\tau-1}}{C \rightarrow B} (1)_{B} \nonumber  
     \end{split} \end{equation}    
     We note, that the equation is uniquely solvable as long as $(I-\PC)$ is invertible,  $(I-\PC) $ is invertible if $\rho (\PC ) <1$. 
     
     By assuming that $\bar{P}_0$ is  irreducible, we will show that for all  $v \in \R^{ |C|}$, $ \lVert v\PC  \rVert_1 <  \lVert  v \rVert_1 $. Since this holds in particular for the eigenvectors, it follows that all eigenvalues $|\lambda|<1$ and thus $\rho (\PC ) <1$.
     
We know that $\PC$ is a substochastic matrix (row sum is $\leq 1$) since it is a product of substochastic matrices, and that all entries are non-negative, i.e., $   \sum_j |\PC_{ij}|  \leq 1$ for all $i\in C$.

Moreover, there exists at least one row with row sum less than 1 since by irreducibility of  $\bar{P}_0$, there must be at least one state $i$ in $C$ with a positive probability to go to $A$ or $B$. We call the first one of this kind by $i^*$ with  $   \sum_{j} |\PC_{i^*j}|    < 1$. 
Thus we can compute
        \begin{equation}  \begin{split}
         \lVert v \PC  \rVert_1  
         &= \lVert \sum_i v_i  \PC_{ij} \rVert_1  = \sum_j \sum_i |v_i| |\PC_{ij}|  =  \sum_i  |v_i| \left( \sum_j |\PC_{ij}| \right)  \nonumber \\
          &=
       \left(  \sum_{i\neq i^*}  |v_i| \sum_{j} |\PC_{ij}| \right) + |v_{i^*}| \sum_{j} |\PC_{i^*j}| \leq \sum_{i\neq i^*}  |v_i| +  |v_{i^*}| \sum_{j} |\PC_{i^*j}| < \sum_{i }  |v_i| = \lVert  v \rVert_1.
       \nonumber
        \end{split} \end{equation}
We have shown that a unique solution which we will call $q^+_0$ exists. The forward committors for $m=1,\dots, M-1$ can uniquely be computed therof by using \eqref{eq:q_f_periodic}.

For the case of the backward committor, we can proceed analogously 
by using the
 time-reversed transition probabilities mapping $M$ instances back $\bar{\transback_m}:= \transback_m \cdot \transback_{m-1} \cdots \transback_{m-M+1}$, and by noting that if $\bar{\trans}_0$ is irreducible, also  $\bar{\transback_0}$ is irreducible, which follows from the definition of irreducibility of $\bar{\trans}_0$ and using \eqref{eq:periodic_backward}.

    \end{proof}
    
\subsubsection{Theorem \ref{Thm:conservationlaw_periodic}}
\begin{proof} 
To show that the flux conservation in node $i\in C$ holds, we compute
    \begin{equation*}
         \sum_{j\in \St}  \left( f_{m,ij}^{AB}-f_{m-1,ji}^{AB}\right) 
         \overset{(1)}{=}  \pi_{m,i} q^-_{m,i} \sum_{j\in\St } \trans_{m,ij} q^+_{m+1,j} - \pi_{m,i} q^+_{m,i} \sum_{j\in\St }  q^-_{m-1,j} \trans^-_{m,ij}  \overset{(2)}{=} 0   
    \end{equation*}
    using (1) $\trans^-_{m,ij}\pi_{m,i} = \trans_{m-1,ji}\pi_{m-1,j}$ and (2)  the backward and forward committor equations for $i\in C$.
    
    Next we want to show that the current of reactive trajectories leaving $A$ during one period equals the current entering $B$ during one period. We calculate
 \begin{equation}  \begin{split}
    \sum_{m\in\M} \left( \sum_{\substack{i\in \St \\ j \in \St}} \current_{m,ij} \right) =\sum_{m\in\M}  \left( \sum_{\substack{i\in A \\ j \in \St}} \current_{m,ij}  +\sum_{\substack{i\in C \\ j \in \St}} \current_{m,ij}  +\sum_{\substack{i\in B\\ j \in \St}}  \underbrace{\current_{m,ij}}_{ = 0}  \right) =  \sum_{m\in\M} \left( \sum_{\substack{i\in \St \\ j \in A}} \underbrace{\current_{m,ij}}_{ = 0} +  \sum_{\substack{i\in \St \\ j \in C}} \current_{m,ij}  + \sum_{\substack{i\in \St\\ j \in B}} \current_{m,ij} \right)\nonumber 
\end{split} \end{equation}
using that $\current_{m,ij}=0$ if $i\in B, j\in \St$ and if $i\in \St, j\in A$. 
And by the current conservation for $i\in C$, $m\in\M$ and by relabeling $i,j,m$,
$$\sum_{m\in\M} \sum_{\substack{i\in C \\ j \in \St}} \current_{m,ij} = \sum_{m\in\M} \sum_{\substack{i\in C \\ j \in \St}} \current_{m-1, ji}=\sum_{m\in\M}  \sum_{\substack{j\in C \\ i \in \St}} \current_{m,ij}$$
we arrive at
\begin{equation}  \begin{split}
     \sum_{m\in\M}  \left( \sum_{\substack{i\in A \\ j \in \St}} \current_{m,ij}  +\sum_{\substack{i \in \St\\j\in C}} \current_{m,ij} \right)  = \sum_{m\in\M} \left( \sum_{\substack{i\in \St \\ j \in C}} \current_{m,ij}  +\sum_{\substack{i\in \St\\ j \in B}} \current_{m,ij}  \right) \nonumber 
 \end{split} \end{equation}
 implying that $ \sum_{m\in\M}   \sum_{\substack{i\in A \\ j \in \St}} \current_{m,ij}   =   \sum_{m\in\M}   \sum_{\substack{i\in \St\\ j \in B}} \current_{m,ij}.$
    
\end{proof}
\subsubsection{Theorem \ref{Thm:q_fb_finite}}
\begin{proof}
First, we recall what we have already seen in the proof of Theorem \ref{Thm:q_fb_periodic}. If $i \in A$ for any $n \in \{0, \dots, N-1\}$ we have that $q_i^+(n) = 0$, $q_i^-(n) = 1$ and if $i \in B$ for any $n \in \{0, \dots, N-1\}$ we have that $q_i^+(n) = 1$, $q_i^-(n) = 0$. 

Second, we find a final condition for the forward committor on $i \in C$
\[ q_i^+(N-1) = \prob(\tau_B^+(N) < \tau_A^+(N) | X_N = i) = \prob(X_N \in B | X_N = i) \overset{(1)}{=} 0 \] 
and an initial condition for the backward committor on $i \in C$
\[ q_i^-(0) = \prob(\tau_A^-(0) > \tau_B^-(0) | X_0 = i) = \prob(X_0 \in A | X_0 = i) \overset{(2)}{=} 0 , \]
where we have used in (1) and (2) that $i \in C$.

Third, by following the same arguments used to prove \eqref{eq:q_f_periodic_proof} (in the proof of Theorem \ref{Thm:q_fb_periodic}) we get for any $n \in \{0, \dots, N-2\}$ that
\[ q_i^+(n) = \sum\limits_{j \in \St} \, \trans_{ij}(n) \, q_j^+(n+1) . \]
Analogously, for any $n \in \{1, \dots, N-1\}$ we get that
\[ q_i^-(n) = \sum\limits_{j \in \St} \, \transback_{ij}(n) \, q_j^-(n-1). \]
\end{proof}

\subsubsection{Theorem \ref{Thm:conservationlaw_finite}}
\begin{proof} 
First, for any $i\in C$ and $n \in \{1, \dots, N-2\}$ we have on one hand that
\begin{equation} \label{eq:calc_conservation_finite_1}
\begin{split}
\sum_{j\in\St} \current_{ij}(n)
&=  q_i^-(n)  \dist{i}{n}  \Big( \sum_{j\in\St} \trans_{ij}(n) q^+_j(n+1) \Big) \\
&\overset{(1)}{=}   q_i^-(n) \dist{i}{n}  q^+_i(n)
\end{split} 
\end{equation}
and on the other hand that
\begin{equation} \label{eq:calc_conservation_finite_2}
\begin{split}
\sum_{j\in\St} \current_{ji}(n-1) 
&= \sum_{j\in\St} q_j^-(n-1) \dist{j}{n-1}  \trans_{ji}(n-1) q^+_i(n) \\
&\overset{(2)}{=}  q^+_i(n)  \dist{i}{n}  \Big( \sum_{j\in\St} \transback_{ij}(n) q_j^-(n-1) \Big) \\
&\overset{(3)}{=}  q^+_i(n)  \dist{i}{n}  q_i^-(n) ,
\end{split} 
\end{equation}
where (1) and (3) follow by \eqref{eq:q_f_finite} and \eqref{eq:q_b_finite} and (2) follows by \eqref{eq:transback_finite}. 

Second, by using that $\current_{ij}(n)=0$ for any $n \in \{0, \dots, N-2\}$ if $i\in B, j\in \St$ and if $i\in \St, j\in A$ we arrive at the following equality
\begin{equation}
\begin{split}
\sum\limits_{n=0}^{N-2} \left( \sum_{\substack{i\in \St \\ j \in \St}} \current_{ij}(n) \right) 
= \sum\limits_{n=0}^{N-2} \left( \sum_{\substack{i\in A \\ j \in \St}} \current_{ij}(n) +\sum_{\substack{i\in C \\ j \in \St}} \current_{ij}(n)  \right) 
= \sum\limits_{n=0}^{N-2} \left( \sum_{\substack{i\in \St \\ j \in C}} \current_{ij}(n) + \sum_{\substack{i\in \St\\ j \in B}} \current_{ij}(n) \right) . \nonumber 
\end{split}
\end{equation}
Then, we show that
\begin{equation}
\begin{split}
\sum\limits_{n=0}^{N-2} \sum_{\substack{i\in C \\ j \in \St}} \current_{ij}(n)
&= \sum_{\substack{i\in C \\ j \in \St}} \current_{ij}(0) + \sum\limits_{n=1}^{N-2} \sum_{\substack{i\in C \\ j \in \St}} \current_{ij}(n) \overset{(4)}{=} \sum\limits_{n=1}^{N-2} \sum_{\substack{i\in C \\ j \in \St}} \current_{ji}(n-1) \\
&\overset{(5)}{=} \sum\limits_{n=0}^{N-3} \sum_{\substack{j\in C \\ i \in \St}} \current_{ij}(n) + \sum_{\substack{j\in C \\ i \in \St}} \current_{ij}(N-2) = \sum\limits_{n=0}^{N-2} \sum_{\substack{j\in C \\ i \in \St}} \current_{ij}(n) ,
\end{split}
\end{equation}
where in (4) we have applied the time-dependent current conservation for $i \in C$, ${n \in \{1, \dots, N-2\}}$ and we have used that ${f^{AB}_{ij}(0)=0}$, and in (5) we have relabeled $i, j$ and used that ${f^{AB}_{ij}(N-2)=0}$. As a consequence 
\[ \sum\limits_{n=0}^{N-2}  \sum_{\substack{i\in A \\ j \in \St}} \current_{ij}(n) = \sum\limits_{n=0}^{N-2} \sum_{\substack{i\in \St\\ j \in B}} \current_{ij}(n) . \]
\end{proof}
\end{document}